\documentclass[reqno, oneside]{amsart}
\usepackage{amsfonts, amsmath, amsthm, amssymb}
\usepackage[shortlabels]{enumitem}
\usepackage{geometry}
\usepackage[utf8]{inputenc}
\usepackage{nicefrac}
\usepackage{tikz-cd}

\usepackage{adjustbox}
\usepackage{xspace}
\usepackage{etoolbox}
\usepackage{booktabs}
\usepackage[hidelinks]{hyperref}
\usepackage[capitalise]{cleveref}
\usepackage[colorinlistoftodos]{todonotes}

\newtheorem{theorem}{Theorem}[section]
\newtheorem{lemma}[theorem]{Lemma}
\newtheorem{proposition}[theorem]{Proposition}
\newtheorem{corollary}[theorem]{Corollary}

\theoremstyle{definition}
\newtheorem{definition}[theorem]{Definition}
\newtheorem{remark}[theorem]{Remark}

\newtheorem*{claim}{Claim}
\newtheorem{example}[theorem]{Example}

\crefformat{footnote}{#2\footnotemark[#1]#3}

\newcommand{\ReDeclareMathOperator}[2]{\let#1\relax\DeclareMathOperator{#1}{#2}}

\DeclareMathOperator{\at}{\textit{at}}
\DeclareMathOperator{\pt}{\textit{pt}}

\DeclareMathOperator{\RO}{RO}
\DeclareMathOperator{\GO}{AO}
\DeclareMathOperator{\RC}{RC}
\ReDeclareMathOperator{\int}{int}
\ReDeclareMathOperator{\O}{O}
\renewcommand{\O}{\texorpdfstring{\operatorname{O}}{O}}
\ReDeclareMathOperator{\C}{C}
\ReDeclareMathOperator{\K}{K}
\ReDeclareMathOperator{\S}{Sat}
\ReDeclareMathOperator{\CS}{CSat}
\ReDeclareMathOperator{\KS}{KS}
\ReDeclareMathOperator{\LC}{LC}
\ReDeclareMathOperator{\LO}{CLC}
\ReDeclareMathOperator{\CSAT}{CSAT}
\ReDeclareMathOperator{\SFilt}{SFilt}
\ReDeclareMathOperator{\WLC}{WLC}
\ReDeclareMathOperator{\WLO}{WCLC}
\ReDeclareMathOperator{\GC}{AC}
\ReDeclareMathOperator{\P}{\mathcal{P}}

\newcommand{\categoryname}[1]{\ensuremath{\mathbf{#1}}\xspace}
\newcommand{\DeclareCategory}[2]{\newcommand{#1}{\categoryname{#2}}}
\DeclareCategory{\SMT}{SMT}
\DeclareCategory{\KMT}{KMT}
\DeclareCategory{\MT}{MT}
\DeclareCategory{\Frm}{Frm}
\DeclareCategory{\BFrm}{BoolFrm}
\DeclareCategory{\CFrm}{CFrm}
\DeclareCategory{\SFrm}{SFrm}
\DeclareCategory{\SHFrm}{SHFrm}
\DeclareCategory{\NFrm}{NormFrm}
\DeclareCategory{\SNFrm}{SNormFrm}
\DeclareCategory{\HFrm}{HFrm}
\DeclareCategory{\RegFrm}{RegFrm}
\DeclareCategory{\SRegFrm}{SRegFrm}
\DeclareCategory{\CRegFrm}{CRegFrm}
\DeclareCategory{\SCRegFrm}{SCRegFrm}
\DeclareCategory{\KRegFrm}{KRegFrm}
\DeclareCategory{\Top}{Top}
\DeclareCategory{\Sob}{Sob}
\DeclareCategory{\IA}{IA}
\DeclareCategory{\Reg}{Reg}
\DeclareCategory{\Set}{Set}
\DeclareCategory{\CABA}{CABA}
\DeclareCategory{\HA}{HA}
\DeclareCategory{\cHA}{cHA}

\newcommand{\sq}{\underline{\mkern-1mu\square\mkern-1mu}\mkern2mu}
\newcommand{\s}{\square}
\renewcommand{\diamond}{\lozenge}
\let\ampersand\&
\renewcommand{\&}{\mathbin{\ampersand}}
\newcommand{\upset}{\mathord{\uparrow}}
\newcommand{\Q}{\mathbb Q}
\newcommand{\N}{\mathbb N}

\setlist[enumerate]{font=\normalfont}

\newcommand{\llcurly}{\mathrel{\lhd\hspace{-1pt}\lhd}}

\tikzset{
  symbol/.style={
    draw=none,
    every to/.append style={
      edge node={node [sloped, allow upside down, auto=false]{$#1$}}}
  }
}

\makeatletter
\patchcmd{\@setaddresses}{\indent}{\noindent}{}{}
\patchcmd{\@setaddresses}{\indent}{\noindent}{}{}
\patchcmd{\@setaddresses}{\indent}{\noindent}{}{}
\patchcmd{\@setaddresses}{\indent}{\noindent}{}{}
\makeatother

\title{McKinsey-Tarski Algebras: An alternative pointfree approach to topology}
\author{Guram Bezhanishvili and Ranjitha Raviprakash}
\address{\newline
Department of Mathematical Sciences\newline
New Mexico State University\newline
Las Cruces, NM 88003\newline
USA\newline}

\email{guram@nmsu.edu}
\email{prakash2@nmsu.edu}


\begin{document}

\begin{abstract}

McKinsey and Tarski initiated the study of interior algebras. We propose complete interior algebras as an alternative pointfree approach to topology. We term these algebras McKinsey-Tarski algebras or simply MT-algebras. Associating with each MT-algebra the lattice of its open elements defines a functor from the category of MT-algebras to the category of frames, which we study in depth. We also study the dual adjunction between the categories of MT-algebras and topological spaces, and show that MT-algebras provide a faithful generalization of topological spaces. Our main emphasis is on developing a unified approach to separation axioms in the language of MT-algebras, which generalizes separation axioms for both topological spaces and frames.

\end{abstract}

\maketitle

\tableofcontents

\section{Introduction}

The aim of pointfree topology is to study topological spaces by means of their lattices of open sets. This results in the category $\Frm$ of frames and its dual category $\bf Loc$ of locales, which are the main subject of study (see \cite{Johnstone,PicadoPultr2012,PicadoPultr2021}). Associating with each topological space the lattice of open sets yields a contravariant functor from the category $\Top$ of topological spaces to $\Frm$. The contravariant functor in the other direction is obtained by working with the space of points (completely prime filters) of a frame. This results in a contravariant adjunction between $\Top$ and $\Frm$, which restricts to a dual equivalence between the reflective subcategories of spatial frames (lattices of open sets) and sober spaces (spaces of points). For details see \cite[Ch.~II]{PicadoPultr2012} and \cref{sec: Prelims} below.

It is well known that topological spaces can alternatively be studied by means of interior operators on their powerset algebras. McKinsey and Tarski \cite{McKinseyTarski} initiated the study of interior operators on an arbitrary Boolean algebra. This has resulted in the notion of an interior algebra, that is a Boolean algebra equipped with a unary function satisfying Kuratowski's axioms for interior.\footnote{McKinsey and Tarski mainly worked with closure operators on Boolean algebras and the corresponding closure algebras. Rasiowa and Sikorski \cite{Rasiowasikorski} made the switch to interior operators, and the name interior algebra was coined by Blok \cite{BlokThesis}.} The main goal of McKinsey and Tarski was to ``set up the foundation of a new algebraic calculus, which could be regarded as a sort of algebra of topology" and to ``study both the internal algebraic properties of this calculus and its relation to topology as ordinarily conceived" (\cite[p.~141]{McKinseyTarski}). 

It turned out that interior algebras and Heyting algebras of their open elements are at the heart of the G\"odel translation of intuitionistic logic into Lewis' modal system $\sf S4$ (see, e.g., \cite{McKinseyTarski1948,Rasiowasikorski,BlokThesis,RedModalLogic,Esakia}). This has triggered an extensive study of interior algebras and their connection to Heyting algebras. However, less emphasis was placed on the study of algebraic properties of those interior algebras that arise as powersets of topological spaces. Other than the early works of Rasiowa and Sikorski (in the 1950s and 1960s), we are only aware of the 1990 PhD of Naturman \cite{naturman1990interior}, where a dual equivalence between $\Top$ and the category of complete and atomic interior algebras is established, which is a direct generalization of Tarski duality between the categories of sets and complete and atomic Boolean algebras (see \cref{sec: Prelims,sec MT-algebras}). 

Complete and atomic interior algebras can be thought of as the generalization of spatial frames. Indeed, each spatial frame is isomorphic to the lattice of open elements of a complete and atomic interior algebra. In order to account for all frames, we need to drop the atomicity assumption. We thus arrive at the main object of study of this paper: complete interior algebras, which we term McKinsey-Tarski algebras or simply MT-algebras. We show that each frame is isomorphic to the lattice of open elements of an appropriate MT-algebra. Thus, MT-algebras can indeed be thought of as the generalization of frames, and taking the lattice of open elements defines an essentially surjective functor from the category $\MT$ of MT-algebras to $\Frm$. 

In analogy with frames, we call atomic MT-algebras spatial. The dual adjunction between $\Top$ and $\Frm$ generalizes to a dual adjunction between $\Top$ and $\MT$, which restricts to a dual equivalence between $\Top$ and the category $\SMT$ of spatial MT-algebras. Thus, the language of MT-algebras allows us to capture all topological spaces. In contrast, in the language of frames we are only able to capture sober spaces.

The bulk of the paper is dedicated to the study of topological separation axioms in the language of MT-algebras. Each separation axiom is described by identifying an appropriate subset of an MT-algebra that join-generates it. We also connect separation axioms for MT-algebras with the corresponding separation axioms for frames, thus providing a unifying approach to separation axioms both in the topological and pointfree settings.

The paper is organized as follows. In \cref{sec: Prelims}, we recall the dual adjunction between the categories of topological spaces and frames, which restricts to a dual equivalence between the categories of sober spaces and spatial frames. We also discuss how this yields Tarski duality between the categories of complete and atomic Boolean algebras and sets. \cref{sec MT-algebras} introduces MT-algebras, which is our main subject of study, and discusses their relationship with topological spaces and frames. \cref{sec: Functor O} examines the properties of the functor $\O:\MT \to \Frm$. In \cref{sec:5}, we define $T_0$-algebras and sober MT-algebras and study the restriction of $\O$ to these subcategories. In addition, we study when the space of atoms of an MT-algebra $M$ is homeomorphic to the space of points of the frame $\O(M)$, from which we derive that $\O$ yields an equivalence between the categories of spatial sober MT-algebras and spatial frames. 

\cref{sec:6} studies $T_{1/2}$-algebras and $T_1$-algebras. We prove that $M$ is a $T_{1/2}$-algebra iff $M$ is isomorphic to the MacNeille completion of the Boolean envelope of $\O(M)$, and that if $M$ is a sober $T_{1/2}$-algebra, then $M$ is spatial iff $\O(M)$ is a spatial frame. We also generalize the corresponding result for spaces by showing that $M$ is a $T_1$-algebra iff $M$ is a $T_{1/2}$-algebra and $\O(M)$ is a subfit frame. In \cref{sec:7}, we introduce Hausdorff and regular MT-algebras. We show that each Hausdorff MT-algebra is sober, thus generalizing the corresponding result for spaces. We also prove that if $M$ is a Hausdorff MT-algebra, then $\O(M)$ is a Hausdorff frame, from which we derive that $\O$ yields an equivalence between the categories of spatial Hausdorff MT-algebras and spatial Hausdorff frames, which further restricts to an equivalence between the categories of spatial regular MT-algebras and spatial regular frames. 

Finally, \cref{sec:8} explores completely regular and normal MT-algebras. We show that a $T_1$-algebra is completely regular iff $\O(M)$ is a completely regular frame, from which we derive that the categories of spatial completely regular MT-algebras and spatial completely regular frames are equivalent. Similar results are also proved for normal MT-algebras and normal frames. In addition, we prove a version of Urysohn's lemma for MT-algebras, from which we derive that each normal MT-algebra is completely regular. The paper concludes with several open problems.

\section{Preliminaries} \label{sec: Prelims}

In this section we recall the well-known dual adjunction between frames and topological spaces, which restricts to a dual equivalence between spatial frames and sober spaces. Restricting to Boolean frames (complete Boolean algebras) then yields a dual equivalence between complete and atomic Boolean algebras (CABAs) and sets, known as Tarski duality.

We recall (see, e.g., \cite[p.~10]{PicadoPultr2012}) that a {\em frame} is a complete lattice $L$ satisfying the join-infinite distributive law
$$
a\wedge\bigvee S=\bigvee\{a\wedge s\mid s\in S\}
$$
for all $a\in L$ and $S\subseteq L$. The order-dual of a frame is a {\em co-frame}. Thus, a co-frame is a complete lattice $L$ satisfying the meet-infinite distributive law
$$
a\vee\bigwedge S=\bigwedge\{a\vee s\mid s\in S\}
$$
for all $a\in L$ and $S\subseteq L$.

A {\em frame homomorphism} is a map $h:L \to M$ between frames preserving finite meets and arbitrary joins. Let $\Frm$ be the category of frames and frame homomorphisms. 

We also let $\Top$ be the category of topological spaces and continuous maps. There is a well-known dual adjunction between $\Top$ and $\Frm$ (see, e.g., \cite[p.~16]{PicadoPultr2012}). It is described as follows. The contravariant functor $\Omega :\Top \to \Frm$ sends a space $X$ to the frame $\Omega(X)$ of opens of $X$ and a continuous map $f:X\to Y$ to the frame homomorphism $\Omega(f):\Omega(Y)\to\Omega(X)$ given by $\Omega(f)(U)=f^{-1}[U]$ for each $U\in\Omega(Y)$. 

To describe the functor in the other direction, we recall that a {\em point} of a frame $L$ is a completely prime filter of $L$. The contravariant functor $\pt:\Frm \to \Top$ then sends each frame $L$ to the space $\pt(L)$ of points of $L$. The topology of $\pt(L)$ is given by $\zeta[L]$ where $\zeta(a)=\{ p\in\pt(L) \mid a\in p\}$ for each $a\in L$. Moreover, $\pt$ sends a frame homomorphism $h:L \to M$ to the continuous map $\pt(h):\pt(M)\to\pt(L)$ given by $\pt(h)(p)=h^{-1}[p]$ for each $p\in\pt(M)$. For $X\in\Top$ and $x\in X$, let $\delta(x)=\{ U\in\Omega(X) \mid x\in U \}$. We then have:

\begin{theorem} [see, e.g., {\cite[p.~17]{PicadoPultr2012}}]
$(\Omega, \pt)$ is a dual adjunction whose units are $\delta : 1_\Top \to \pt\circ\Omega$ and $\zeta : 1_\Frm \to \Omega\circ\pt$.
\end{theorem}

This dual adjunction yields a dual equivalence between the full subcategories of $\Top$ and $\Frm$ which we describe next. Recall 
that a frame $L$ is {\em spatial} if $ a \leq b$ implies that there is $y \in \pt(L)$ such that $ a \in y$ and $ b \notin y$. Let $\SFrm$ be the full subcategory of $\Frm$ consisting of spatial frames. 
We also recall that a closed set in a topological space $X$ is {\em irreducible} if it is not the union of two proper closed sets, and that $X$ is {\em sober} if each irreducible closed set is the closure of a unique point. Let $\Sob$ be the full subcategory of $\Top$ consisting of sober spaces.

We have that $\zeta : L \to \Omega(\pt(L))$ is an isomorphism iff $L$ is spatial, and that $\delta : X \to \pt(\Omega(X))$ is a homeomorphism iff $X$ is sober. We thus arrive at the following well-known theorem.

\begin{theorem} [see, e.g., {\cite[pp.~18, 20]{PicadoPultr2012}}] \label{thm: DP duality}
The dual adjunction between $\Top$ and $\Frm$ restricts to a dual equivalence between $\Sob$ and $\SFrm$.

\begin{center}
\begin{tikzcd}
    \Top \ar[r, bend left] & \Frm \ar[l, bend left]\\
    \Sob \ar[r, <->] \ar[u, symbol=\leq] & \SFrm \ar[u, symbol=\leq]
\end{tikzcd}
\end{center}
\end{theorem}

In the diagram above,  
\begin{tikzpicture}
\draw[<-] (0,-.05) to [bend right] (.5,-.05);
\draw[->] (0,.05) to [bend left] (.5,.05);
\end{tikzpicture} 
stands for ``dual adjunction," $\longleftrightarrow$ for ``dual equivalence," and $\leq$ for ``a full subcategory of." For $X\in\Top$, it is common to call the sober space $\pt(\Omega(X))$ the {\em soberification} of $X$. Also, for $L\in\Frm$, the spatial frame $\Omega(\pt(L))$ is called the {\em spatialization} of $L$.

We recall that a frame $L$ is {\em Boolean} provided $L$ is a Boolean algebra. Therefore, Boolean frames are exactly complete Boolean algebras. Let $\BFrm$ be the full subcategory of $\Frm$ consisting of Boolean frames. Observe that frame homomorphisms between Boolean frames are exactly complete Boolean homomorphisms. Thus, $\BFrm$ is the category of complete Boolean algebras and complete Boolean homomorphisms.

If $L$ is a Boolean frame, then every point of $L$ is of the form ${\uparrow}x$ for a unique atom $x$ of $L$. Therefore, there is a bijection between points and atoms of $L$. Thus, $L$ is spatial iff $L$ is atomic. Consequently, the category of spatial Boolean frames is exactly the category $\CABA$ of complete atomic Boolean algebras and complete Boolean homomorphisms. 

Let $L\in\BFrm$. If $x$ is an atom of $L$, then $\zeta(x)= \{\upset{x}\}$, so $\pt(L)$ is a discrete space. We identify the category of discrete spaces with the category $\Set$ of sets and functions, which we view as a full subcategory of $\Top$. The restriction of $\Omega$ to $\Set$ is then simply the contravariant powerset functor $\P:\Set\to\CABA$ which sends a set $X$ to its powerset $\P(X)$ and a function $f:X\to Y$ to the complete Boolean homomorphism $\P(f)=f^{-1}$.

On the other hand, if $h : L \to M$ is a frame homomorphism between Boolean frames, then $h$ is a complete Boolean homomorphism. Therefore, for each $x\in\at(M)$ we have $h^{-1}(\upset{x}) = \upset{h^*(x)}$, where $h^* : M \to L$ is the left adjoint of $h$ (given by $h^*(x)=\bigwedge\{a \in M \mid x \leq h(a)\}$). Thus, we can think of the restriction of $\pt$ to $\BFrm$ as the functor $\at:\BFrm\to\Set$, which associates with each $L\in\BFrm$ the set $\at(L)$ of atoms of $L$, and with each complete Boolean homomorphism $h:L\to M$ the function $\at(h):\at(M)\to\at(L)$ defined by $\at(h)(x)=h^*(x)$



Consequently, the dual adjunction $(\Omega,\pt)$ between $\Top$ and $\Frm$ restricts to the dual adjunction $(\P,\at)$ between $\Set$ and $\BFrm$, and the duality 
between $\Sob$ and $\SFrm$ restricts to Tarski duality between $\Set$ and $\CABA$:

\begin{theorem}[Tarski duality]
	\Set and \CABA are dually equivalent.
\end{theorem}

In fact, the composition $\P\circ\at:\BFrm\to\CABA$ is a reflector.

\begin{remark}\label{Tarski}
The units of the dual adjunction $(\P,\at)$ are given by $\eta:1_{\BFrm}\to\P\circ\at$ and $\varepsilon:1_{\Set}\to\at\circ\P$, where for a Boolean frame $B$ and a set $X$, we define $\eta_B : B \to \P(\at(B))$ and $\varepsilon_X : X \to \at(\P(X))$ by 
\[
\eta_{B} (a) =\{ x \in \at(B) \mid x \leq a \} \quad \mbox{and} \quad \varepsilon (x) =\{ x \}.
\]
We then have that $\varepsilon_X$ is a bijection for each $X\in\Set$, $\eta_B$ is an onto complete Boolean homomorphism for each $B\in\BFrm$, and $\eta_B$ is an isomorphism iff $B \in \CABA$.
\end{remark}

\section{McKinsey-Tarski algebras} \label{sec MT-algebras}

In this section we introduce McKinsey-Tarski algebras---the main subject of our study. For this we first recall interior algebras and their connection to Heyting algebras. We then define the category $\MT$ of MT-algebras and study how it relates to the categories $\Top$ and $\Frm$. 
We show that there is a functor $\O:\MT\to\Frm$. We also show that there is a dual adjunction between $\MT$ and $\Top$, which restricts to a dual equivalence between $\Top$ and the reflective subcategory $\SMT$ of $\MT$ consisting of spatial MT-algebras. This duality generalizes Tarski duality for $\CABA$ to $\SMT$.

We start by recalling the well-known definition of an interior algebra \cite{McKinseyTarski}:

\begin{definition} 
Let $B$ be a Boolean algebra. An {\em interior operator} on $B$ is a unary function $\square : B \to B$ satisfying Kuratowski's axioms 
for all $a,b\in B$:
\begin{itemize}
    
    \item $\square 1=1$.
    \item $\square(a \wedge b) = \square a \wedge \square b$. 
    \item $\square a \leq a$. 
    \item $\square a \leq \square \square a$. 
\end{itemize}

An {\em interior algebra} is a pair $(B,\square)$ where $B$ is a Boolean algebra and $\square$ is an interior operator on~$B$.
\end{definition}

\begin{remark}
Each interior operator has the corresponding {\em closure operator} $\diamond=\lnot\square\lnot$, and interior algebras can equivalently be defined as pairs $(B,\diamond)$ where $\diamond$ is a closure operator on $B$. This is the approach originally taken by McKinsey and Tarski \cite{McKinseyTarski}, who termed these algebras {\em closure algebras}. In their influential book \cite{Rasiowasikorski}, Rasiowa and Sikorski mainly worked with the interior operator and referred to these algebras as {\em topological Boolean algebras}.
  The term ``interior algebra" was coined in \cite{BlokThesis}. In modal logic these algebras are also known as {\em $\sf S4$-algebras} since they model the well-known modal system $\sf S4$ (see, e.g., \cite[p.~214]{RedModalLogic}).
\end{remark}

\begin{definition} \cite[p.~146]{McKinseyTarski}
Let $A$ be an interior algebra. 
\begin{enumerate}
\item We call $a \in A$ {\em open} if $a= \square a$. Let $\O(A)$ be the collection of open elements of $A$.
\item We call $a\in A$ {\em closed} if $a=\diamond a$. Let $\C(A)$ be the collection of closed elements of $A$. 
\end{enumerate}
\end{definition}

It is well known (see, e.g., \cite[Prop.~2.2.4]{Esakia}) that $\O(A)$ is a bounded sublattice of $A$ that forms a Heyting algebra, where the Heyting implication is given by $a\to b=\square(\lnot a \vee b)$ for all $a,b\in\O(A)$.  
Dually, $\C(A)$ is a bounded sublattice of $A$ that forms a co-Heyting algebra, where the co-implication is given by $a\leftarrow b=\diamond(b \wedge \lnot a)$ for all $a,b\in\C(A)$ \cite[p.~130]{McKinseyTarski1946}.

In fact, every Heyting algebra is isomorphic to the one of the form $\O(A)$ for some interior algebra $A$. Similarly, every co-Heyting algebra is isomorphic to $\C(A)$ for some interior algebra $A$. 
This can be seen by utilizing the well-known Boolean envelope construction. Recall (see, e.g., \cite[Sec.~V.4]{DL}) that the {\em Boolean envelope} of a bounded distributive lattice $L$ is a pair $(B(L), e)$, where $B(L)$ is a Boolean algebra and $e: L \to B(L)$ is a bounded lattice embedding satisfying the following universal mapping property: for any Boolean algebra $A$ and a bounded lattice homomorphism $h: L \to A$, there is a unique Boolean homomorphism $B(h): B(L) \to A$ such that $B(h)\circ e = h$. 
 
\[\begin{tikzcd}
L \ar[r, "e"] \ar[d,"h"'] & B(L) \ar[dl,dashed,"B(h)"]\\
A
\end{tikzcd}\]
The inclusion map $e: L \to B(L)$ has a right adjoint $\square: B(L) \to L$, and by identifying $L$ with $e[L]$ we get that $(B(L),\square)$ is an interior algebra (see, e.g., \cite[Sec.~2.5]{Esakia}).

\begin{example} \label{example: mt-algebra} \label{example: inverse cts is mt}
Standard examples of interior algebras come from topology. If $X$ is a topological space, then $(\P(X),\int)$ is obviously an interior algebra. Moreover, if $f:X \to Y$ is a continuous function between topological spaces, then $f^{-1}:\P(Y) \to \P(X)$ is a complete Boolean homomorphism satisfying $f^{-1}(\int(A)) \subseteq \int(f^{-1}(A))$.
\end{example}

This example motivates the following key definition.

\begin{definition}
$ $
\begin{enumerate}
\item We call an interior algebra $(B,\square)$ a {\em McKinsey-Tarski algebra} or an {\em MT-algebra} if $B$ is a complete Boolean algebra.
\item  An {\em MT-morphism} between MT-algebras $M$ and $N$ is a complete Boolean homomorphism $h : M\to N$ such that $h(\square_M a) \leq \square_N h(a)$ for each $ a \in M$.
\item Let $\MT$ be the category of MT-algebras and MT-morphisms. (Composition is usual function composition and identity morphisms are identity maps.) 
\end{enumerate}
\end{definition}

\vspace{2mm}

The following is then immediate:

\begin{theorem} \label{thm: Top to MT}
The assignment $X \mapsto (\P(X),\int)$ and $f \mapsto f^{-1}$ defines a contravariant functor $\P : \Top \to \MT$.
\end{theorem}

We next show that the assignment $M \mapsto \O(M)$ extends to a functor from $\MT$ to $\Frm$. As we already pointed out, $\O(M)$ is a bounded sublattice of $M$. Moreover, for $S\subseteq\O(M)$, we have $\square(\bigvee S)=\bigvee S$. Thus, $\O(M)$ is a frame.
Dually, $\C(M)$ is a 
co-frame. We thus arrive at the following:

\begin{theorem} \label{thm: MT to Frm}
If $M$ is an MT-algebra, then $\O(M)$ is a frame and $\C(M)$ is a co-frame.
\end{theorem}

\begin{remark} \label{square and diamond as adjoints}
$ $
\begin{enumerate}[ref=\theremark(\arabic*)]
\item For each $a\in M$ we have $\square a =\bigvee \{b \in \O(M) \mid b \leq a\}$. Therefore, $\square:M \to \O(M)$ is the right adjoint of the inclusion $\O(M)\hookrightarrow M$. 
\item Similarly, $\diamond a = \bigwedge \{c \in \C(M) \mid a \leq c\}$ for each $a\in M$, and hence $\diamond:M \to \C(M)$ is the left adjoint of the inclusion $\C(M)\hookrightarrow M$. \label[remark]{square and diamond as adjoints-2} 
\end{enumerate}
\end{remark}

 
\begin{lemma} \label{lemma: restriction mt is frame}
If $ h: M\to N$ is an MT-morphism, then its restriction $ h| _{\O(M)} : \O(M)\to\O(N)$ is a frame homomorphism.
\end{lemma}

\begin{proof}
Since $h$ preserves arbitrary joins and finite meets, it is sufficient to show that the restriction $h|_{\O(M)}$ is well defined. Let $a \in \O(M)$. We have 
\[
h(\s_M a) \leq \s_N h(a) \leq h(a) =h(\s_M a).
\] 
Therefore, $h(\s_M a) = \s_N h(a)$, and hence $h(\s_{M}a) \in \O(N)$. 
\end{proof}

As an immediate consequence of \cref{thm: MT to Frm,lemma: restriction mt is frame}, we obtain:

\begin{theorem}
$\O : \MT \to \Frm $ is a covariant functor.
\end{theorem}

\begin{corollary} \label{cor: OP=Omega}
$ \O \circ \P =\Omega$.
\begin{center}
\begin{tikzcd}[row sep=2em]
      \bf{MT}\arrow{rr}{\O}  & & \Frm \\  
      & \Top \arrow{ul} {\P} \arrow[ur, "\Omega"'] 
      \end{tikzcd}
\end{center}
\end{corollary}

 
We now show that the functor $\at:\BFrm\to\Set$ extends to a functor $\at:\MT\to\Top$. Let $M \in \MT$, $X = \at(M)$, and recall that 
$\eta : M \to \P(X)$ is given by 
$
\eta (a) =\{ x \in X \mid x \leq a \}.
$
By \cref{Tarski}, $\eta$ is an onto complete Boolean homomorphism. Therefore, the restriction of $\eta$ to $\O(M)$ is a frame homomorphism. Hence, the image $\tau:=\eta[\O(M)]$ is a topology on $X$. We thus have:

\begin{lemma} \label{lem: MT to Top}
Let $ M \in \MT$. Then $(\at(M),\tau) \in \Top$ and $\eta: M\to\P(\at(M))$ is an onto MT-morphism.
\end{lemma}

\begin{proof}
We already observed that $(\at(M),\tau) \in \Top$ and that $\eta$ is an onto Boolean homomorphism. Moreover, since $\eta (\square a )$ is open, from $\eta (\square a )\subseteq \eta (a)$ it follows that $\eta (\square a )\subseteq \int(\eta (a))$ for each $a\in M$. Thus, $\eta$ is an onto MT-morphism. 
\end{proof}

\begin{remark} \label{rem: closed sets}
The co-frame of closed sets of $(\at(M),\tau)$ is $\eta[\C(M)]$. Also, since $\eta$ is an MT-morphism, we obtain that 
$
\overline{\eta(a)}=(\int(\eta(\neg a)))^c \subseteq (\eta ( \square \neg a))^c = \eta(\diamond a)
$
for each $a\in M$.
\end{remark}

\begin{lemma}\label{thm: at(h) is well defined}
    If $h : M \to N$ is an MT-morphism, then $\at(h) : \at(N) \to \at(M)$ is a continuous map.
\end{lemma}

\begin{proof}

We recall from Tarski duality that the left adjoint $h^*:N\to M$ of $h$ restricts to a map $\at(h):\at(N)\to\at(M)$. Therefore, it is sufficient to show that $[\at(h)]^{-1}\eta_M(\square_M a) = \eta_N(h(\square_M a))$ for each $a \in M$. We have 
\begin{align*}
    x \in [\at(h)]^{-1}\eta_M(\square_M a) 
    &\iff \at(h)(x) \in \eta_M(\square_M a)\\
    &\iff h^\ast(x) \leq \square_M a\\
    &\iff x \leq h(\square_M a)\\
    &\iff x \in \eta_N(h(\square_M a)).
\end{align*}
Thus, $\at(h)$ is continuous.
\end{proof}

As an immediate consequence of \cref{lem: MT to Top,thm: at(h) is well defined}, we obtain:

\begin{theorem} \label{thm: MT to Top}
$\at : \MT \to \Top$ is a contravariant functor.
\end{theorem}

\begin{lemma} \label{lem: epsilon homo}
    Let $X\in\Top$. Then $\varepsilon : X \to \at (\P(X))$ is a homeomorphism.
\end{lemma}

\begin{proof}
     
    We recall that $\varepsilon : X \to \at (\P(X))$ is given by $\varepsilon(x) = \{x\}$ for each $x\in X$. It follows from Tarski duality that $\varepsilon$ is a bijection. We show that $\varepsilon$ is a homeomorphism. 
    Let $U$ be an open subset of $X$, so an open element of  $(\P(X),\int)$. We have 
      \begin{equation*}
        x \in \varepsilon^{-1}(\eta(U)) 
        \iff \varepsilon(x) \in \eta(U)
        \iff \{x\} \subseteq U
        \iff x \in U.
    \end{equation*}
    Therefore, $\varepsilon^{-1}(\eta(U)) = U$, and so $\varepsilon$ is continuous. Moreover, since $\varepsilon$ is a bijection, the previous identity implies $\varepsilon(U)=\eta(U)$, and hence $\varepsilon$ is open. Thus, $\varepsilon$ is a homeomorphism.   
 \end{proof}
As an immediate consequence of  
\cref{thm: Top to MT,thm: MT to Top}, and \cref{lem: MT to Top,lem: epsilon homo}, the dual adjunction $(\P,\at)$ between $\Set$ and $\BFrm$ lifts to $\Top$ and $\MT$, and we obtain:

\begin{theorem} \label{thm: dual adj MT and Top}
$(\P, \at)$ is a dual adjunction between $\Top$ and $\MT$ with units $\varepsilon : 1_{\Top} \to \at\circ\P$ and $\eta : 1_{\MT} \to \P\circ\at$.
\end{theorem}

\begin{lemma}\label{preserves interior}
Let $M$ be an MT-algebra.
\begin{enumerate}[ref=\thelemma(\arabic*)]
\item $\eta:M\to\P(\at(M))$ is one-to-one iff $M$ is atomic. 
\item If $M$ is atomic, then $\eta(\square a)=\int(\eta(a))$ for each $a\in M$.\label[lemma]{lem: eta}
\end{enumerate}
\end{lemma}

\begin{proof} 
(1) This follows from Tarski Duality.

(2) Let $a\in M$. For $x \in \at(M)$ we have
\begin{align*}
    x \in \int(\eta(a))
    &\iff \exists b \in M : x \in \eta(\square b) \subseteq \eta(a)\\ 
    &\iff x \leq \square b \leq a && \text{by (1)}\\
    &\iff x \leq \square b \leq \square a\\ 
    &\iff x \in \eta(\square a). 
\end{align*}
Thus, $\eta(\square a)=\int(\eta(a))$.
    \end{proof}
    
    \begin{remark}
    It is immediate from \cref{lem: eta} that if $M$ is atomic, then $\eta(\diamond a) = \overline{\eta(a)}$ for each $a \in M$. 
    \end{remark}
    
    \begin{example} \label{exa: simple box}
 
We show that \cref{lem: eta} may fail if $M$ is not atomic. We recall (see \cite [p.~221]{Halmos1956}) that an interior operator $\s$ on a Boolean algebra $B$ is \emph{simple} if 
$$\s a =\begin{cases}
    1 &\text {if } a=1\\
    0 &\text {otherwise}.
    \end{cases}$$
Let $B$ be a complete atomless Boolean algebra, $2=\{0,1\}$ the two-element Boolean algebra, $M=B \times 2$, and $\square$ a simple interior operator on $M$.
Then $M \in \MT$ and $x=(0,1)$ is the only atom of $M$, so $\at(M)=\{x\}$. We have $\eta(\square x)=\varnothing$ since $\square x=0$, but $\eta(x)=\at(M)$, and so $\int(\eta(x))=\at(M)$.
\end{example}

\begin{definition}
We call atomic MT-algebras {\em spatial}. Let $\SMT$ be the full subcategory of $\MT$ consisting of spatial MT-algebras.
\end{definition}
    
    \begin{theorem} \label{thm: SMT dual to Top}
    The dual adjunction of \cref{thm: dual adj MT and Top} restricts to a dual equivalence between $\Top$ and $\SMT$.
    \end{theorem}

\begin{proof} 
It is sufficient to observe that $\varepsilon : X \to \at(\P(X))$ is a homeomorphism (see \cref{lem: epsilon homo}) and that $\eta : M \to \P(\at(M))$ is an MT-isomorphism iff $M \in \SMT$ (see \cref{preserves interior}).
\end{proof}

\begin{remark}
The above theorem extends Tarski duality by incorporating interior operators in the signature of CABAs. It goes at least as far back as \cite[Thm.~2.1.7]{naturman1990interior}.
\end{remark}

As a consequence of \cref{thm: SMT dual to Top}, we obtain that the composition $\P\circ\at:\MT\to\SMT$ is a reflector. Thus, we have:

\begin{theorem}
\SMT is a reflective subcategory of \MT.
\end{theorem}

\section{The functor \texorpdfstring{$\O:\MT\to\Frm$}{O : MT -> Frm}}
\label{sec: Functor O}

In this section we show that the functor $\O:\MT\to\Frm$ is essentially surjective, but that it is neither full nor faithful. We also compare the space of atoms of an MT-algebra $M$ to the space of points of the frame $\O(M)$, and show that the restriction of $\O$ to the category of spatial MT-algebras lands in the category of spatial frames.

We start by showing that 
$\O$ is essentially surjective. Let $L\in\Frm$ and let $B(L)$ be the Boolean envelope of $L$. As we saw in the previous section, the inclusion $e:L\to B(L)$ has a right adjoint $\square$ such that $(B(L),\square)$ is an interior algebra and $L$ is (isomorphic to) the fixpoints of $\square$. Since $B(L)$ may not be complete, $(B(L),\square)$ may not be an MT-algebra. Therefore, we take the MacNeille completion $\overline{B(L)}$ of $B(L)$ (see, e.g., \cite[Sec.~XII.3]{DL}). For simplicity, we identify $B(L)$ with its image in $\overline{B(L)}$ and think of it as a subalgebra of $\overline{B(L)}$. Following \cite{MacNeille}, define the lower extension $\sq:\overline{B(L)}\to\overline{B(L)}$ of $\square:B(L)\to B(L)$ by 
\[
\sq x = \bigvee\{\square a\mid a\in B(L) \mbox{ and } a\leq x\}.
\] 
By \cite[Thm.~3.5]{MacNeille}, if $(B,\square)$ is an interior algebra, then so is $\left(\overline{B},\sq\right)$. 
As an immediate consequence, we obtain:

\begin{lemma} \label{lemma: macneille is MT}
$\left(\overline{B(L)},\sq\right) \in \MT$.
\end{lemma}

\begin{theorem} \label{thm: essentially surj}
The functor $\O:\MT\to\Frm$ is essentially surjective.

\end{theorem}

\begin{proof}
We must show that for each $ L \in \Frm $ there is $ M \in \MT$ such that $\O(M)$ is isomorphic to $L$.
Let $L\in\Frm$. By 
\cref{lemma: macneille is MT}, $\left(\overline{B(L)},\sq\right) \in \MT$. By \cite[pp.~91--92]{Gratzer}, the embedding $L \hookrightarrow \overline{B(L)}$ is a frame embedding, and we may identify $L$ with a subframe of $\overline{B(L)}$. Therefore, 
it remains to show that $\O\left(\overline{B(L)},\sq\right)=L$. For this it is sufficient to show that for each $x\in \overline{B(L)}$, we have $x\in L$ iff $x=\sq x$. If $x\in L$, then $x\in \overline{B(L)}$, so $x=\square x = \sq x$. Conversely, if $x = \sq x$, then $x = \bigvee\{\square a\mid a\in B(L) \mbox{ and } a\leq x\} \in L$ since $L$ is a subframe of $\overline{B(L)}$. 
\end{proof}

\begin{remark}
That $L \hookrightarrow \overline{B(L)}$ is a frame embedding is part of the well-known Funayama theorem (see \cite[p.~92]{Gratzer} or \cite[p.~274]{BezhGabelaiaJibladze2013}). 
It follows from the proof above that 
$\sq$ is the right adjoint of the embedding $L \hookrightarrow \overline{B(L)}$.
\end{remark}

We next show that the assignment $L \mapsto \overline{B(L)}$ does not induce a functor from $\Frm$ to $\MT$.

\begin{example} \label{exa: not a functor}
    Let $\tau$ be the topology of cofinite subsets of $\mathbb N$ together with $\varnothing$, and let $L$ be the frame of opens of $(\mathbb N,\tau)$.  
    By the proof of \cref{thm: essentially surj}, $L$ is isomorphic to $\O\left(\overline{B(L)}, \sq\right)$. Observe that $B(L)$ is isomorphic to the Boolean algebra of finite and cofinite subsets of $\mathbb N$, and hence $\overline{B(L)}$ is isomorphic to $\P(\mathbb N)$. Let $2=\{0,1\}$ and define a frame homomorphism $h : L \to 2$ by 
\[
	h(U) = \begin{cases}
		1 &\text{if } U \neq \varnothing\\
		0 &\text{if } U = \varnothing
	\end{cases}
\]     Then $B(h) : B(L) \to 2$ satisfies $B(h)(U) = 0$ iff $U$ is finite. Since $\mathbb N$ is the union of its finite subsets, no complete Boolean homomorphism $g : \overline{B(L)} \to 2$ can extend $B(h)$. Since $2$ is isomorphic to $\overline{B(2)}$, we conclude that $h:L\to 2$ cannot be extended to a complete Boolean homomorphism $g:\overline{B(L)}\to\overline{B(2)}$. 
    
\end{example}

\begin{lemma}
The functor $\O$ is neither full nor faithful.
\end{lemma}

\begin{proof}	
That $\O$ is not full follows from \cref{exa: not a functor}. Indeed, $\O\left(\overline{B(L)}\right)=L$, $\O\left(\overline{B(2)}\right)=2$, but the frame homomorphism $h:L\to 2$ cannot be extended to an MT-morphism from $\overline{B(L)}$ to $\overline{B(2)}$.

To see that $\O$ is not faithful,	let $B = \{0, a, b, 1\}$ be the four-element Boolean algebra and let $\square : B \to B$ be the simple interior operator on $B$ (see \cref{exa: simple box}). We let $i : B \to B$ be the identity and define $h : B \to B$ by $h(0)=0$, $h(1)=1$, $h(a) = b$, and $h(b) = a$. Then $i$ and $h$ are distinct MT-morphisms. But $\O(i) = \O(h)$ since they coincide on $\O(B)=\{0,1\}$. 
\end{proof}

\begin{remark}
We will see that $\O$  becomes full and/or faithful on several natural subcategories of $\MT$.
\end{remark}

We now turn our attention to the comparison of the space of atoms of an MT-algebra $M$ to the space of points of the frame $\O(M)$. Define $\vartheta : \at(M) \to \pt(\O(M))$ by $\vartheta(x)= \upset x \cap \O(M)$. It is elementary to see that $\vartheta$ is well defined.

\begin{proposition}\label{continuous}
The map $\vartheta$ is continuous.
\end{proposition}

\begin{proof}
It is sufficient to show that for each $a \in \O(M)$ we have $\vartheta ^{-1}(\zeta(a))=\eta(a)$. Indeed, for $x\in\at(M)$,\[
x \in \vartheta ^{-1}(\zeta(a)) \iff \vartheta(x) \in \zeta(a) \iff a \in \vartheta(x) \iff 
x \leq a \iff x \in \eta(a).
\]
Thus, $\vartheta$ is continuous.
\end{proof}

In general, $\vartheta$ is neither one-to-one nor onto.

\begin{example}\label{theta}
$ $

\begin{enumerate}
    \item To see that $\vartheta$ is not one-to-one, let $X$ be a set with the trivial topology. Then $\at(\P(X))$ is homeomorphic to $X$. On the other hand, $\Omega(X)=\{\varnothing , X\}$, and hence $\pt(\Omega(X))$ is a singleton. Thus, if $X$ has more than one point, then $\vartheta$ can't be one-to-one.
    \item To see that $\vartheta$ is not onto, let $M$ be a complete atomless Boolean algebra and let $\square : M \to M$ be the simple interior operator on $M$ (see \cref{exa: simple box}). Then $M$ is an MT-algebra such that $\at(M)=\varnothing$. On the other hand, $\O(M) = \{0,1\}$, and hence $\pt(\O(M))$ is a singleton. Thus, $\vartheta$ can't be onto. 
\end{enumerate} 
\end{example}

\begin{proposition}\label{one-one}
Let $M\in\MT$. Then $\vartheta$ is one-to-one iff $\at(M)$ is a $T_0$-space.
\end{proposition}

\begin{proof}
We first observe that  for $x\in\at(M)$ and $a\in\O(M)$ we have  
\begin{equation}
a\in\vartheta(x) \iff x\leq a \iff x\in\eta (a). \label{varthetaandeta} \tag{\dag}
\end{equation}
Now, suppose that $\vartheta$ is one-to-one. Let $x,y\in\at(M)$ with $x \neq y$. Then $\vartheta(x) \neq \vartheta(y)$. Without loss of generality we may assume that $\vartheta(x) \not\subseteq \vartheta(y)$, so there is $a \in \vartheta(x)$ such that $a \not \in \vartheta(y)$. Therefore, (\ref{varthetaandeta}) implies that $x \in \eta(a)$ and $y \not \in \eta(a)$. Thus, $\at(M)$ is a $T_0$-space.

Conversely, suppose that $\at(M)$ is a $T_0$-space and $x,y \in \at(M)$ with $x\ne y$. Without loss of generality we may assume that there is $a\in\O(M)$ such that $x \in \eta(a)$ and $y \not \in \eta(a)$. Therefore, (\ref{varthetaandeta}) implies that $a \in \vartheta(x)$ and $a \not \in \vartheta(y)$. Thus, $\vartheta$ is one-to-one.
\end{proof}

In the next section we characterize when $\vartheta$ is onto and when $\vartheta$ is a homeomorphism. 
We conclude this section by showing that the restriction of $\O$ to $\SMT$ lands in $\SFrm$.

\begin{proposition}\label{SMT implies SFrm}
If $M\in\SMT$, then $\O(M)\in\SFrm$.
\end{proposition}

\begin{proof}
Suppose that $M$ is a spatial MT-algebra. Let $a,b \in \O(M)$ with $a \nleq b$. Since $M$ is atomic, there is $x\in\at(M)$ such that $x \leq a$ but $x\nleq b$. Therefore, $a \in \vartheta(x)$ but $b \notin \vartheta(x)$. Thus, $\vartheta(x) \in \zeta(a)$ but $\vartheta(x)\notin\zeta(b)$. Consequently, $\zeta(a)\not\subseteq \zeta(b)$, and so $\O(M)$ is a spatial frame.
\end{proof}

\section{\texorpdfstring{$T_0$}{T0} and sober MT-algebras} \label{sec:5}

In this section we start developing the separation axioms for MT-algebras that generalize the corresponding separation axioms for topological spaces. We begin by introducing $T_0$, weakly sober, and sober MT-algebras. We show that the restriction of $\O:\MT\to\Frm$ to the category of $T_0$-algebras is faithful, and characterize when $\vartheta:\at(M)\to\pt(\O(M))$ is onto and a homeomorphism. From this we derive that the restriction of $\O$ is an equivalence between the categories $\SMT_{\mathrm{Sob}}$ of spatial sober MT-algebras and $\SFrm$ of spatial frames. We also show that the dual adjunction and dual equivalence between $\Top$, $\MT$, and $\SMT$ restricts to the categories of $T_0$ and sober MT-algebras.

We start by studying the $T_0$ separation in MT-algebras. Let $X$ be a topological space. We recall that $A\subseteq X$ is {\em saturated} if $A$ is an intersection of open sets, and that $A$ is {\em locally closed} if it is an intersection of an open set and a closed set. These have an obvious generalization to MT-algebras:

\begin{definition}\label{def:WLC}
Let $M\in\MT$. 
\begin{enumerate}
\item An element $a\in M$ is \emph{saturated} if $a$ is a meet of open elements.
Let $\S(M)$ be the collection of saturated elements of $M$. 
\item Dually, call $a\in M$ \emph{co-saturated} if $a$ is a join of closed elements.
Let $\CS(M)$ be the collection of co-saturated elements of $M$. 
\item An element $a \in M$ is \emph{locally closed} if it is a meet of an open element and a closed element.
Let $\LC(M)$ be the collection of locally closed elements of $M$.
\item Dually, call $a \in M$ \emph{co-locally closed} if  
it is a join of a closed element and an open element. Let $\LO(M)$ be the collection of co-locally closed elements of $M$.
\item An element $a \in M$ is \emph{weakly locally closed} if 
it is a meet of a saturated element and a closed element. Let $\WLC(M)$ be the collection of weakly locally closed elements of $M$.
\item Dually, call $a \in M$ \emph{weakly co-locally closed} if 
it is a join of a co-saturated element and an open element. Let $\WLO(M)$ be the collection of weakly co-locally closed elements of $M$.
\end{enumerate}.
\end{definition}

We recall that a subset $S$ of a complete lattice $L$ \emph{join-generates} 
$L$ if each element of $L$ is a join of elements from $S$. Dually, $S$ \emph{meet-generates} $L$
if each element of $L$ is a meet of elements from $S$.

\begin{definition}\label{def:T0}
We call $M\in\MT$ a {\em $T_0$-algebra} if $\WLC(M)$ join-generates $M$. 
\end{definition}

\begin{remark}
Equivalently, $M\in\MT$ is a $T_0$-algebra provided $\WLO(M)$ meet-generates $M$.
\end{remark}

Let $\MT_0$ be the full subcategory of \MT consisting of $T_0$-algebras. Let also $\Top_0$ be the full subcategory of $\Top$ consisting of $T_0$-spaces.

\begin{lemma} \label{lem: Top0}
    Let $X\in \Top$ and $M \in \MT$. 
    \begin{enumerate}[ref=\thelemma(\arabic*)]
    		\item $X\in\Top_0$ iff $(\P(X),\int)\in\MT_0$. \label[lemma]{lem: char T0}
    		\item $M\in\MT_0$ implies $\at(M)\in\Top_0$. \label[lemma]{T_0}
    \end{enumerate}
\end{lemma}

\begin{proof}
(1) Observe that $X$ is a $T_0$-space iff $\{x\} = \overline{\{x\}} \cap \bigcap \{U \in \Omega(X) \mid x \in U\}$ for each $x\in X$. Therefore, since each $A \subseteq X$ is a union of singletons, we obtain that $X\in\Top_0$ iff $ (\P(X),\int)\in\MT_0$.

(2) Let $x,y \in \at(M)$ be distinct. Since $M\in\MT_0$, we must have $x \in \WLC(M)$. Therefore, $x = \bigwedge S \wedge b$, where $S \subseteq \O(M)$ and $b \in \C(M)$. 
    Since distinct atoms are non-comparable, we have $y \nleq x$, so $y\nleq a$ for some $a\in S$ or $y\nleq b$. Thus, $x \in \eta(a)$ and $y \notin \eta(a)$, or else $\neg b \in \O(M)$ with $y \le \neg b$ and $x\not\le \neg b$, yielding $y\in \eta(\neg b)$ and $x \not \in \eta(\neg b)$. Consequently, $\at(M) \in \Top_0$.
\end{proof}

\begin{remark} \label{rem: connection to Raney}
	Every element $a$ of a $T_0$-algebra $M$ can be expressed as $a = \bigvee_{i} (\bigwedge S_i \wedge \neg u_i)$ for some $S_i \subseteq \O(M)$ and $u_i \in \O(M)$. As a consequence, $\O(M)$ generates $M$ as a complete Boolean algebra. If $M$ is spatial, then $M$ is completely distributive (see, e.g., \cite [p.~123]{HalmosBA}), so the converse is also true. Thus, a spatial MT-algebra $M$ is a $T_0$-algebra iff $\O(M)$ generates $M$ as a complete Boolean algebra (cf. \cite[p.~964]{RaneyAlgebra}).
\end{remark}

\begin{example}\label{nonspatial}
For trivial reasons, there exist non-spatial $T_0$-algebras. Indeed, each atomless MT-algebra $M$ in which $\square$ is identity is such. This example  serves as a generic example of a non-spatial MT-algebra that satisfies all the separation axioms that will be considered in this paper.  
\end{example}

Let $\SMT_0 = \SMT\cap\MT_0$. Combining \cref{thm: dual adj MT and Top,thm: SMT dual to Top,lem: Top0}, we obtain:

\begin{theorem} \label{thm: MT0 and Top0}
    The dual adjunction between $\Top$ and $\MT$ restricts to a dual adjunction between $\Top_0$ and $\MT_0$, and the dual equivalence between $\Top$ and $\SMT$ restricts to a dual equivalence between $\Top_0$ and $\SMT_0$.
\end{theorem}

We show that the restriction of the functor $\O$ to $\MT_0$ is faithful.

\begin{theorem} \label{thm: O faithful for T0}
    The restriction $\O:\MT_0\to\Frm$ is faithful.
\end{theorem}

\begin{proof}
    Let $h,g :M\to N$ be MT-morphisms between MT-algebras. Suppose that $h|_{\O(M)}= g|_{\O(M)}$
    and $a \in M$. If $M\in\MT_0$, then $a=\bigvee T$, where for each $t\in T$ we have $t=\bigwedge S_t \wedge b_t$ with $S_t\subseteq\O(M)$ and $b_t\in\C(M)$. Therefore, writing $b_t = \neg c_t$ for some $c_t \in \O(M)$, we obtain
    \begin{eqnarray*}
    h(a) &=& h\left(\bigvee T\right) = \bigvee \{ h(t) \mid t\in T \} = \bigvee \left\{ h\left(\bigwedge S_t \wedge \neg c_t\right) \mid t \in T \right\} \\
    &=& \bigvee \left\{ \bigwedge \{ h(a) \mid a \in S_t \} \wedge \neg h(c_t) \mid t \in T \right\} = \bigvee \left\{ \bigwedge \{ g(a) \mid a \in S_t \} \wedge \neg g(c_t) \mid t\in T \right\} \\
    &=& \bigvee \left\{ g\left(\bigwedge S_t \wedge \neg c_t\right) \mid t\in T \right\} = \bigvee \{ g(t) \mid t \in T \} = g\left(\bigvee T\right) = g(a).
    \end{eqnarray*} 
    Thus, $h=g$, and hence the restriction of $\O$ to $\MT_0$ is faithful.
\end{proof}

We next turn our attention to sober MT-algebras. We recall that an element $a\ne 0$ of a lattice $L$ is \emph{join-irreducible} if it is not the join of two strictly smaller elements. Dually, $a\ne 1$ is \emph{meet-irreducible} if it is not the meet of two strictly bigger elements.

\begin{definition}
We call an MT-algebra $M$ {\em weakly sober} 
if for each join-irreducible element $p$ of $\C(M) $ there is $x \in \at(M)$ such that $p = \diamond x$. 
\end{definition} 

\begin{remark}
Equivalently, an MT-algebra $M$ is weakly sober provided for each meet-irreducible element $m$ of $\O(M) $ there is $x \in \at(M)$ such that $m = \neg\diamond x$.

\end{remark}

\begin{definition} \label{MTsob}
We call an MT-algebra $M$ {\em sober} 
if $M$ is a weakly sober $T_0$-algebra. 
\end{definition}

\begin{lemma}\label{unique}
If $M$ is sober, then for each join-irreducible $p\in\C(M)$ there is a unique $x\in\at(M)$ such that $p = \diamond x$
\end{lemma}

\begin{proof}
Let $q\in\C(M)$ be join-irreducible. Since $M$ is weakly sober, $p=\diamond x$ for some $x\in\at(M)$. To see that $x$ is unique, let $\diamond x=\diamond y$ for some $y\in\at(M)$. Then $\eta(\diamond x) = \eta(\diamond y)$. By \cref{square and diamond as adjoints-2}, 
\begin{equation}
\eta(\diamond x) = \eta\left(\bigwedge \{c \in \C(M) \mid x \leq c \}\right) = \bigcap \{\eta(c) \mid x \leq c \in \C(M) \}= \overline{\{x\}}. 
\label{uniqueness} \tag{\dag \dag}
\end{equation}
Similarly, $\eta(\diamond y) = \overline{\{y\}}$. Thus, $\overline{\{x\}} = \overline{\{y\}}$, and hence $x=y$ since $\at(M)$ is a $T_0$-space by \cref{T_0}. 
\end{proof}

\begin{remark} \label{rem: spatial sober}
\hfill
\begin{enumerate}[ref=\theremark(\arabic*)]
\item \cref{unique} can equivalently be formulated as follows: If $M$ is sober, then for each meet-irreducible $m\in\O(M)$ there is a unique $x\in\at(M)$ such that $m = \neg\diamond x$. 

\item If $M$ is spatial, the converse of \cref{unique} is also true. For this it is sufficient to see that $M$ is a $T_0$-algebra. Let $a \neq 0$. Since $M$ is spatial,  there is an atom $x \leq a$. We show that 
\[
x = \diamond x \wedge \bigwedge \{ \neg \diamond y \mid x \leq \neg \diamond y \}.
\] 
Clearly $x \leq \diamond x \wedge \bigwedge \{ \neg \diamond y \mid x \leq \neg \diamond y \}$. If $  \diamond x \wedge \bigwedge \{ \neg \diamond y \mid x \leq \neg \diamond y \} \nleq x$, then there is $z \in \at(M)$ such that $z \leq \diamond x \wedge \bigwedge \{ \neg \diamond y \mid x \leq \neg \diamond y \}$ but $z \not\leq x$. From $z \leq \diamond x$ it follows that $\diamond z \leq \diamond x$. Since $x$ is the unique atom whose closure is $\diamond x$ and $z\ne x$, we must have $\diamond x \not \leq \diamond z$, and hence $x \not \leq \diamond z$. Therefore, $x \leq \neg \diamond z$, and so $z \leq \neg \diamond z$, which is a contradiction because $z \neq 0$. Thus, $M$ is a $T_0$-algebra.\label[remark]{rem:Spatialsober}

\item \cref{exa: simple box} shows that the converse of \cref{unique} may fail for non-spatial $M$. Indeed, the MT-algebra $M$ in that example is a non-spatial MT-algebra whose only atom is $(0,1)$ and $\diamond (0,1)=(1,1)$. Thus, each join-irreducible of $\C(M)$ is determined by a unique atom of $M$,  
but $M$ is not $T_0$.

\end{enumerate}

\end{remark}

Let $\MT_{\bf Sob}$ be the full subcategory of $\MT_0$ consisting of sober algebras. Also recall that $\Sob$ is the full subcategory of $\Top_0$ consisting of sober spaces.

\begin{lemma} \label{lem: Sober}
Let $ X \in \Top$ and $M \in \MT$. 
\begin{enumerate}[ref=\thelemma(\arabic*)]
\item $X\in\Sob$ iff $ (\P(X),\int)\in \MT_{\bf Sob}$. \label[lemma]{lem: Sob}
\item $M\in\MT_{\bf Sob}$ implies $\at(M)\in\Sob$. \label[lemma]{MT_sober}
\end{enumerate}
\end{lemma}

\begin{proof}
(1) This follows from \cref{unique,rem:Spatialsober}.

(2) Let $F$ be an irreducible closed set in $\at(M)$. Since $F$ is closed, $F=\eta(a)$ for some $a\in \C(M)$ (see \cref{rem: closed sets}). 
Let $b = \diamond \bigvee \eta(a)$. We show that $\eta(a) = \eta(b)$. Clearly $\bigvee \eta(a) \leq a$, so $b=\diamond \bigvee \eta(a) \leq a$ because $a$ is closed. Therefore, $\eta(b) \subseteq \eta(a)$. For the other inclusion, suppose $x \in \eta(a)$. Then $x \leq \bigvee \eta(a)$, which implies that $x \leq \diamond \bigvee \eta(a) = b$. Thus, $\eta(a) \subseteq \eta(b)$, hence the equality. Clearly $b\in\C(M)$. We prove that $b$ is join-irreducible in $\C(M)$. 
Suppose $c,d \in \C(M)$ with $b \le c \vee d$. Then $\eta(b) \subseteq \eta(c) \cup \eta(d)$. Since $\eta(b) = \eta(a)$ is irreducible, $\eta(b) \subseteq \eta(c)$ or $\eta(b) \subseteq \eta(d)$. Without loss of generality we may assume that $\eta(b) \subseteq \eta(c)$. Then $b = \diamond \bigvee \eta(a) \leq \diamond \bigvee \eta(c) \leq c$. Therefore, $b$ is join-irreducible.
Since $M\in\MT_{\bf Sob}$, we have $b = \diamond x$ for a unique $x \in \at(M)$. Thus, $F=\eta(b)=\overline{\{x\}}$ by (\ref{uniqueness}). Consequently, $\at(M)\in \bf Sob$. 
\end{proof}

\cref{nonspatial} shows that there exist sober MT-algebras that are not spatial. Let $\SMT_{\bf Sob} = \MT_{\bf Sob} \cap \SMT$. Combining \cref{thm: MT0 and Top0,lem: Sober}, we obtain:

\begin{theorem} \label{thm: MTSob and Sob}
    The dual adjunction between $\Top_0$ and $\MT_0$ restricts to a dual adjunction between $\Sob$ and $\MT_{\bf Sob}$, and the dual equivalence between $\Top_0$ and $\SMT_0$ restricts to a dual equivalence between $\Sob$ and $\SMT_{\bf Sob}$.
\end{theorem}

Since $\Sob$ is dually equivalent to $\SFrm$, it follows from \cref{thm: MTSob and Sob} that $\SMT_{\bf Sob}$ is equivalent to $\SFrm$. In fact, it is the functor $\O$ that establishes this equivalence. 
To see this, we first characterize when $\vartheta : \at(M) \to \pt\O(M)$ is onto and when it is a homeomorphism, as promised at the end of \cref{sec: Functor O}.

\begin{lemma}\label{varthetaonto}
Let $M\in\MT$. 
\begin{enumerate}[ref=\thelemma(\arabic*)]
\item $\vartheta$ is onto iff $M$ is weakly sober.
\item $\vartheta$ is a homeomorphism if $M$ is weakly sober and $\at(M)$ is a $T_0$-space.  \label[lemma]{part 2}
\end{enumerate}
\end{lemma}

\begin{proof}
(1) Suppose $M$ is weakly sober. If $p \in \pt(\O(M))$, then $m:=\bigvee \{a \in \O(M) \mid a \notin p\}$ is meet-irreducible in $\O(M)$ (see, e.g., \cite [p.~14]{PicadoPultr2012}). Since $M$ is weakly sober, there is $x\in\at(M)$ such that $m = \neg \diamond x$.  
For $a\in\O(M)$, we have: 
\[
a \notin p \iff a \leq \neg \diamond x \iff \diamond x \leq \neg a \iff x \leq \neg a \iff x \nleq a \iff a \notin \upset{x}.
\]
Thus, $\vartheta(x)=\upset{x}\cap\O(M)=p$, and hence $\vartheta$ is onto.

Conversely, suppose $\vartheta$ is onto. If $m \in \O(M)$ is meet-irreducible, then $p:=\{ a \in \O(M) \mid a \nleq m \}$ is a completely prime filter of $\O(M)$ (see, e.g., \cite [p.~14]{PicadoPultr2012}). Since $\vartheta$ is onto, $p = \upset{x} \cap \O(M)$ for some $x \in \at(M)$. Therefore, for $a \in \O(M)$, we have 
\[
a \leq m \iff a \notin p \iff x \nleq a \iff x \leq \neg a \iff \diamond x \leq \neg a \iff a \leq \neg \diamond x.
\] 
Thus, $m=\neg \diamond x$, and hence $M$ is weakly sober.

(2) If $\vartheta$ is a homeomorphism, then $M$ is weakly sober by (1) and $\at(M)$ is a $T_0$-space by \cref{one-one}. Conversely, if $M$ is weakly sober and $\at(M)$ is a $T_0$-space, then $\vartheta$ is a continuous bijection by (1) and \cref{one-one}. It is left to show that $\vartheta^{-1}$ is also continuous. For this it suffices to show that $\vartheta(\eta(a)) = \zeta(a)$ for all $a \in \O(M)$. Let $y \in \vartheta(\eta(a))$. Then there is $x \in \eta(a)$ such that $y = 
\upset x \cap \O(M)$. Since $x \leq a$, we have $a \in y$, and so $y \in \zeta(a)$. Conversely, let $y \in \zeta(a)$. Then $a \in y$. Since $\vartheta$ is onto, there is $x\in\at(M)$ such that $y = 
\upset x \cap \O(M)$. Therefore, $x \leq a$, so $x \in \eta(a)$, and hence $y \in \vartheta(\eta(a))$. Thus, 
$\vartheta$ is a homeomorphism. 
\end{proof}

\begin{remark} \label{rem: sober implies theta homeo}
If $M$ is a sober MT-algebra, then $M$ is weakly sober and $T_0$, hence $\at(M)$ is a $T_0$-space by \cref{T_0}. Therefore, $\vartheta$ is a homeomorphism by \cref{part 2}.
\end{remark}

\begin{theorem} \label{thm: full}
    The restriction $\O:\SMT_{\bf Sob}\to\SFrm$ is an equivalence.
\end{theorem}

\begin{proof}

By \cref{SMT implies SFrm}, the restriction $\O :\SMT_{\bf Sob} \to \SFrm$ is well defined. It thus suffices to show that $\O$ is essentially surjective, full, and faithful (see \cite[p.~93]{MacLane}). Let $L\in \SFrm$. By \cref{thm: DP duality,cor: OP=Omega}, $L \cong \Omega(\pt(L))=\O(\P(\pt(L))$. Therefore, $\O$ is essentially surjective. Also, $\O$ is faithful by \cref{thm: O faithful for T0}. To see that $\O$ is full, let $M,N\in\SMT_{\bf Sob}$ and $h:\O(M) \to \O(N)$ be a frame homomorphism. Then $\pt(h): \pt(\O(N)) \to \pt(\O(M))$ is a continuous map.
Since $M$ and $N$ are sober, $\vartheta_{M}:\at(M) \to \pt(\O(M))$ and $\vartheta_{N}:\at(N)\to \pt(\O(N))$ are homeomorphisms by \cref{rem: sober implies theta homeo}. Let $f=\vartheta^{-1}_{M}\circ \pt(h) \circ \vartheta_{N}$. Then $f:\at(N) \to \at(M)$ is a continuous map. Therefore, $\P(f): \P(\at(M)) \to \P(\at(N))$ is an MT-morphism. By \cref{thm: SMT dual to Top}, $\eta_{M}:M \to \P(\at(M))$ and $\eta_{N}:N \to \P(\at(N))$ are MT-isomorphisms. Let $g=\eta^{-1}_{N}\circ \P(f) \circ \eta_M$. Then $g$ is an MT-morphism. We show that $g|_{\O(M)}=h$.
\[
\begin{tikzcd}
	\O(M) \ar[r, rightarrow] \ar[d, "h"] & \pt(\O(M)) \ar[r, leftarrow, "\vartheta_{M}" {yshift=3pt}] & \at(M)\ar[r, rightarrow] & \P(\at(M)) \ar[d, "\P(f)"]  \ar[r, leftarrow, "\eta_{M}" {yshift=3pt}] & M \ar[d, "g"] \\
	\O(N) \ar[r, rightarrow]& \pt(\O(N)) \ar[u, "\pt(h)"']  \ar[r, leftarrow, "\vartheta_{N}" {yshift=3pt}] & \at(N) \ar[r, rightarrow] \ar[u, "f"'] & \P(\at(N))  \ar[r, leftarrow, "\eta_{N}" {yshift=3pt}] & N
\end{tikzcd}
\]

\begin{claim}
$\pt(h)^{-1}(\vartheta_{M}(\eta_{M}(a))) = \vartheta_{N}(\eta_{N}(h(a)))$ for each $a \in \O(M)$.
\end{claim} 

\begin{proof}
Since $\vartheta_M$ is a homeomorphism, for each $p \in \pt(\O(M))$ there is a unique $z \in \at(M)$ such that $p=\upset z \cap \O(M)$; and a similar result is true for $\vartheta_N$. Now, let $y \in \pt(\O(N))$. Then
\begin{align*}
	 y \in \pt(h)^{-1}(\vartheta_{M}(\eta_{M}(a))) 
	 &\iff \pt(h)(y) \in \vartheta_{M}(\eta_{M}(a)) \\ 
	 &\iff \pt(h)(y) \in \{\vartheta_{M}(x)\mid x \leq a \} \\
	 &\iff \pt(h)(y)=\upset x \cap \O(M) \mbox{ for a unique } x \leq a \\
	 &\iff a \in \pt(h)(y) \\ 
	 & \iff h(a) \in y \\ 
	 & \iff y = \upset x' \cap \O(N) \mbox{ for a unique } x' \leq h(a) \\
	 &\iff y \in \{\vartheta_{N}(x')\mid x' \leq h(a) \} \\
	 &\iff y \in \vartheta_{N}(\eta_{N}(h(a))). \qedhere
\end{align*}
\end{proof}

Let $a\in\O(M)$. Then  
\begin{align*}
g(a) &= (\eta_N^{-1} \circ \P(f) \circ \eta_M)(a) = (\eta_N^{-1} \circ f^{-1} \circ \eta_M)(a) \\
&= \left(\eta_N^{-1} \circ \left(\vartheta_{M}^{-1}\circ\pt(h)\circ\vartheta_N\right)^{-1} \circ \eta_{M}\right)(a) \\
&= \left(\eta_N^{-1} \circ \vartheta_{N}^{-1} \circ \pt(h)^{-1} \circ \vartheta_{M} \circ \eta_{M}\right)(a) \\
&= \left(\eta_N^{-1} \circ \vartheta_N^{-1} \circ \vartheta_N \circ \eta_N \circ h\right) (a) = h(a), 
\end{align*}
where the second to last equality follows from the Claim. Thus, $\O$ is full. 
\end{proof}

Consequently, we obtain the following diagram of equivalences and dual equivalences that commutes up to natural isomorphisms.

\begin{center}
      \begin{tikzcd}[row sep=2em]
      \bf{SMT_{Sob}}\arrow[->]{rr}{\O}\arrow[shift left=.5ex]{dr}{\at} & & \SFrm \ar[dl, "pt", shift left=.5ex]\\
      & \Sob \arrow[shift left=.5ex]{ul} {\P} \arrow[ur, "\Omega", shift left=.5ex] &
      \end{tikzcd}
\end{center}

\section{\texorpdfstring{$T_{1/2}$ and $T_1$-algebras}{MT1/2 to Top1/2}}\label{sec:6}

In this section we generalize $T_{1/2}$ and $T_1$ separation axioms for topological spaces to MT-algebras. We prove that an MT-algebra $M$ is a $T_{1/2}$-algebra iff $M$ is isomorphic to the MacNeille completion of the Boolean envelope of $\O(M)$. We also prove that a sober $T_{1/2}$-algebra $M$ is spatial iff $\O(M)$ is a spatial frame. Moreover, we show that an MT-algebra $M$ is a $T_1$-algebra iff $M$ is a $T_{1/2}$-algebra and $\O(M)$ is a subfit frame, thus generalizing a similar result for topological spaces. Furthermore, we show that the dual adjunctions and dual equivalences of the previous section restrict to the categories of $T_{1/2}$ and $T_1$-algebras.

Let $X$ be a topological space. We recall that  
$X$ is a a {\em $T_d$-space} or a {\em $T_{1/2}$-space}  
if $\{x\}$ is locally closed for each $x\in X$. It is well known and easy to see that each $T_{1/2}$-space is a $T_0$-space. Let $\Top_{1/2}$ be the full subcategory of $\Top_0$ consisting of $T_{1/2}$-spaces.

\begin{definition} \label{def: T half}
We call $M\in\MT$ a {\em $T_{1/2}$-algebra} if $\LC(M)$ join-generates $M$. 
\end{definition}

\begin{remark}
Equivalently, $M\in\MT$ is a $T_{1/2}$-algebra provided $\LO(M)$ meet-generates $M$.
\end{remark}

Since locally closed elements are weakly locally closed, every $T_{1/2}$-algebra is a $T_0$-algebra. Let $\MT_{1/2}$ be the full subcategory of $\MT_0$ consisting of $MT_{1/2}$-algebras. \cref{nonspatial} shows that there are $T_{1/2}$-algebras that are not spatial. Let $\SMT_{1/2} = \MT_{1/2} \cap \SMT$.

\begin{lemma} \label{lem: Td spaces}
Let $ X \in \Top$ and $M \in \MT$. 
\begin{enumerate} [ref=\thelemma(\arabic*)] 
\item  
$X \in \Top_{1/2}$ iff $(\P(X),\int) \in \MT_{1/2}$. 
\item $M \in \MT_{1/2}$ implies $\at(M) \in \Top_{1/2}$. \label [lemma] {lem: Td alg}
\end{enumerate}
\end{lemma}

\begin{proof}

(1) By definition, $X$ is a $T_{1/2}$-space iff for each $x \in X$ we have $\{x\} = U \cap F$ for some $U$ open and $F$ closed. 
Therefore, since each $A \subseteq X$ is a union of singletons, we obtain that $X\in\Top_{1/2}$ iff $ (\P(X),\int)\in\MT_{1/2}$.

(2) Let $x \in \at(M)$. Since $M \in \MT_{1/2}$, we have  $x = b \wedge c$ for some $b \in \O(M)$ and $c \in \C(M)$. Therefore, $\{x\} = \eta(b) \cap \eta(c)$.   
Since $\eta(b)$ is open and $\eta(c)$ is closed, each point of $\at(M)$ is locally closed. Thus, $\at(M)$ is a $T_{1/2}$-space.
    \end{proof}

As an immediate consequence of \cref{thm: MT0 and Top0,lem: Td spaces}, we obtain:

\begin{theorem} \label{thm: duality for Td}
    The dual adjunction between $\Top_0$ and $\MT_0$ restricts to a dual adjunction between $\Top_{1/2}$ and $\MT_{1/2}$, and the dual equivalence between $\Top_0$ and $\SMT_0$ restricts to a dual equivalence between $\Top_{1/2}$ and $\SMT_{1/2}$.
\end{theorem}

We next give a convenient characterization of $T_{1/2}$-algebras.

\begin{theorem}\label{T_1/2}
An MT-algebra $M$ is a $T_{1/2}$-algebra iff $M$ is isomorphic to $\overline{B(\O(M))}$.
\end{theorem}

\begin{proof}
Let $M$ be an MT-algebra. Then $B(\O(M))$ is isomorphic to a Boolean subalgebra of $M$ (see, e.g., \cite[p.~99]{DL}). Therefore, $M$ is isomorphic to $\overline{B(\O(M))}$ iff $B(\O(M))$ join-generates $M$. Since each element of $B(\O(M))$ can be written as $a=\bigvee(u \wedge \neg v) $, where $u,v \in \O(M)$ (see, e.g., \cite[p.~74]{Rasiowasikorski}), $a=\bigvee(u \wedge c) $, where $u \in \O(M)$ and $c \in \C(M)$. Thus, $B(\O(M))$ join-generates $M$ iff $\LC(M)$ join-generates $M$, and the result follows.
\end{proof}

\begin{corollary}
The restriction  $\O:\MT_{1/2} \to \Frm$ is essentially surjective and faithful.
\end{corollary}

\begin{proof}
For each frame $L$, we have $\O\left(\overline{B(L)}\right)=L$ by the proof of \cref{thm: essentially surj}. By \cref{T_1/2}, $\overline{B(L)}$ is a $T_{1/2}$-algebra. Thus, the restriction $\O:\MT_{1/2} \to \Frm$  is essentially surjective. That it is faithful follows from \cref{thm: O faithful for T0} since each $T_{1/2}$-algebra is a $T_0$-algebra.
 \end{proof}

\begin{theorem} \label{thm: when is M spatial}
        Let $M$ be a sober $T_{1/2}$-algebra. Then $M$ is spatial iff $\O(M)$ is spatial.
    \end{theorem}
    
    \begin{proof}
    Clearly if $M$ is spatial, then so is $\O(M)$ by \cref{SMT implies SFrm}. Conversely, 
        suppose that $\O(M)$ is spatial. Let $0 \neq a \in M$. Since $M$ is a $T_{1/2}$-algebra, $a = \bigvee S$ for some $S \subseteq \LC(M)$. Therefore, $S$ must contain a nonzero element. Thus, there are $u,v \in \O(M)$ such that $0\neq u \wedge \neg v \leq a$. This implies that $u \nleq v$. Since $\O(M)$ is spatial, there is a completely prime filter $p$ of $\O(M)$ such that $u \in p$ and $v \not \in p$. Because $M$ is sober, $p = \upset x \cap \O(M)$ for some $x \in \at(M)$ by \cref{rem: sober implies theta homeo}. Therefore, $x \leq u$ and $x \leq \neg v$. Thus, $x\le a$, and hence $M$ is spatial.
    \end{proof}

\begin{remark} \label{rem: spatiality of M vs O(M)}
It is unclear whether being a $T_{1/2}$-algebra can be dropped from the assumption of \cref{thm: when is M spatial}.
\end{remark}

We now turn our attention to $T_1$-algebras.

\begin{definition} \label{def: T1}
Let $M\in\MT$. We call $M$ a {\em $T_1$-algebra} if $\C(M) $ join-generates $M$. 
\end{definition}

\begin{remark}
Equivalently, $M$ is a $T_1$-algebra provided $\O(M)$ meet-generates $M$.
\end{remark} 

Since closed elements are locally closed, every $T_1$-algebra is a $T_{1/2}$-algebra. Let $\MT_1$ be the full subcategory of $\MT_{1/2}$ consisting of $MT_1$-algebras. \cref{nonspatial}  shows that there are $T_1$-algebras that are not spatial. Let $\SMT_1 = \MT_1 \cap \SMT$ and let $\Top_1$ be the full subcategory of $\Top_{1/2}$ consisting of $T_1$-spaces.

\begin{lemma} \label{lem: T1}
Let $X \in \Top$ and $M \in \MT$. 
\begin{enumerate}
\item $X \in \Top_1$ iff $ (\P(X),\int) \in \MT_1$.
\item $M \in \MT_1$ implies $\at(M) \in \Top_1$.
\end{enumerate}
\end{lemma}

\begin{proof}
(1) We have that $X$ is a $T_1$-space iff $\{x\}$ is closed for each $x\in X$. Therefore, since each $A \subseteq X$ is a union of singletons, $X\in\Top_1$ iff $ (\P(X),\int)\in\MT_1$.

(2) Let $x\in\at(M)$. Since $M$ is a $T_1$-algebra, $x\in\C(M)$. Therefore, $\{x\} = \eta(x)$ is closed. Thus, $\at(M)$ is a $T_1$-space.
\end{proof}

\begin{remark}
	Suppose $M$ is a $T_1$-algebra. Then each $a \in  M $ can be written as $a = \bigwedge S$ for some $S \subseteq \O(M)$. Therefore, $M$ is generated by $\O(M)$ as a  complete lattice. The converse is also true if $M$ is spatial because then $M$ is completely distributive. Thus, a spatial MT-algebra $M$ is a $T_1$-algebra iff $\O(M)$ generates $M$ as a complete lattice (cf.~\cite[p.~964]{RaneyAlgebra}).
\end{remark}

As an immediate consequence of \cref{thm: duality for Td,lem: T1}, we obtain:

\begin{theorem} \label{thm: duality for T1}
    The dual adjunction between $\Top_{1/2}$ and $\MT_{1/2}$ restricts to a dual adjunction between $\Top_1$ and $\MT_1$, and the dual equivalence between $\Top_{1/2}$ and $\SMT_{1/2}$ restricts to a dual equivalence between $\Top_1$ and $\SMT_1$.
\end{theorem}

We next compare our $T_1$-separation to subfitness in pointfree topology. 
We recall (see, e.g., \cite[p.~73]{PicadoPultr2012}) that a frame $L$ is \emph{subfit} if for all $a,b\in L$ we have 
\[
a \nleq b \implies  \exists c \in L : a \vee c = 1 \neq b \vee c.
\]

\begin{proposition} \label{subfit}
 If $M\in\MT_1$, then $\O(M)$ is subfit.
\end{proposition}

\begin{proof}
 Let $a,b \in \O(M)$ with $a \nleq b$. Since $M$ is a $T_1$-algebra, $\neg a = \bigwedge S$ for some $S \subseteq \O(M)$. If $b \vee c=1$ for all $c \in S$, then $\neg a \vee b = (\bigwedge S) \vee b = \bigwedge \{ c \vee b \mid c \in S \} = 1$, so $a \leq b$, a contradiction. Therefore, there is $c \in S$ such that $b \vee c \neq 1$. On the other hand, $a \vee c \ge a \vee \neg a=1$. Thus, $\O(M)$ is subfit.	
\end{proof}

The following example shows that $\O(M)$ being subfit does not imply that $M$ is a $T_1$-algebra.

\begin{example} \label{ex: subfit not T1}
Let $Y=(\N,\tau)$, where $\tau$ is the cofinite topology (see \cref{exa: not a functor}), and let $X$ be the soberification of $Y$. Then $X = (\N \cup\{\omega\},\tau_{\omega})$, where $\tau_{\omega}=\{U \cup \{\omega\}  \mid \ U \subseteq \N \text{ is cofinite} \} \cup \{\varnothing\}$ (see, e.g., \cite[p.~6]{PicadoPultr2012}). Clearly $X$  
    is not $T_1$ because $\{\omega\}$ is not closed. Therefore, $(\P(X),\int)\notin\MT_1$. We show that $\Omega(X)$ is subfit. Let $U,V \in \Omega(X)$ with $U \nsubseteq V$. Then there is $n \in U \setminus V$. Let $W=X\setminus\{n\}$. So $W\in\Omega(X)$, $U \cup W = X$, but $V \cup W\neq X$. Thus, $\Omega(X)$ is subfit.
\end{example}

Nevertheless, $\O(M)$ being subfit makes every $T_{1/2}$-algebra $M$ a $T_1$-algebra. This is reminiscent of the situation in topological spaces (see, e.g., \cite[p.~73]{PicadoPultr2012}).

\begin{theorem}\label{Sfit}
Let $M \in \MT$. Then $M$ is $T_1$ iff $M$ is $T_{1/2}$ and $\O(M)$ is subfit.
\end{theorem}

\begin{proof}

	First suppose that $M$ is $T_1$. Then $M$ is $T_{1/2}$ and it follows from \cref{subfit} that $\O(M)$ is subfit.	Conversely, suppose that $M$ is $T_{1/2}$ and $\O(M)$ is subfit. Let $a \neq 0$. Since $M$ is $T_{1/2}$, there exist $b \in \O(M)$ and $c \in \C(M)$ such that $0 \neq b \wedge c \leq a$. Therefore, $b \nleq \neg c$. By subfitness, there is $u \in \O(M)$ such that $b \vee u = 1$ and $\neg c \vee u \neq 1$. Thus, $0 \neq \neg u \wedge c \leq b \wedge c \leq a$ and $\neg u \wedge c \in \C(M)$. Consequently, $M$ is $T_1$.
\end{proof}

Let $\Frm_{\bf sfit}$ be the full subcategory of $\Frm$ consisting of subfit frames.

\begin{proposition}\label{essentially surjective MT1}
The restriction $\O:\MT_1\to\Frm_{\bf sfit}$ is well defined, essentially surjective, and faithful.
\end{proposition}

\begin{proof}
That the restriction $\O:\MT_1\to\Frm_{\bf sfit}$ is well defined follows from \cref{subfit}, and that it is faithful follows from \cref{thm: O faithful for T0}. Finally, to see that it is essentially surjective, let $L \in \Frm_{\bf sfit}$. By the proof of \cref{thm: essentially surj}, $\O\left(\overline{B(L)}\right)=L$. By \cref{T_1/2}, $\overline{B(L)}$ is a $T_{1/2}$-algebra, so $\overline{B(L)}$ is a $T_1$-algebra by \cref{Sfit}.
\end{proof}

\begin{remark} \label{rem: O not full on MT1}
\cref{exa: not a functor} shows that the restriction $\O:\MT_1\to\Frm_{\bf sfit}$ is not full. 
\end{remark}

We conclude this section by pointing out that the notion of a fit frame, which is a hereditary variant of subfitness in pointfree topology, is not comparable to the $T_1$-separation axiom for MT-algebras. We recall (see, e.g., \cite[p.~74]{PicadoPultr2012}) that a frame $L$ is {\em fit} if for every $a,b \in L$ with $a \nleq b$, there is $c\in L$ such that $a \vee c=1$ and $c\to b \nleq b$. A simple example of an MT-algebra $M$ such that $\O(M)$ is fit but $M$ is not a $T_1$-algebra is given by the four-element Boolean algebra with simple $\square$. An example of a spatial $T_1$-algebra $M$ such that $\O(M)$ is not fit is given by $M=\P(Y)$, where $Y$ is the space considered in \cref {ex: subfit not T1} (see \cite[p.~31]{PicadoPultr2021}).

\section{Hausdorff and regular MT-algebras}\label{sec:7}

In this section we extend the concepts of Hausdorff and regular spaces to MT-algebras. We
show that if $M$ is a Hausdorff MT-algebra, then $\O(M)$ is a Hausdorff frame, but that the converse
is not true in general. On the other hand, we show that a $T_1$-algebra $M$ is regular iff $\O(M)$ is
a regular frame. We also generalize the well-known result that every Hausdorff space is sober to
MT-algebras and restrict the dual adjunctions and dual equivalences of the previous section to the
setting of Hausdorff and regular MT-algebras.

To define Hausdorff MT-algebras, we generalize the well-known notions of regular open and regular closed sets to the setting of MT-algebras.  

\begin{definition}
Let $M \in \MT$. We call $b\in \O(M)$ {\em regular open} if $b=\square\diamond b$, and $c \in \C(M)$ {\em regular closed} if $c=\diamond\square c$. Let $\RO(M)$ and $\RC(M)$ be the collections of regular open and regular closed elements of $M$, respectively.
\end{definition}

It is well known (see, e.g., \cite[p.~ 66]{HalmosBA}) that the set of regular open (resp.~regular closed) subsets of a topological space forms a complete Boolean algebra. It is straightforward to verify that the same holds for MT-algebras. We thus have:

\begin{proposition}
If $M$ is an MT-algebra, then $\RO(M)$ and $\RC(M)$ are complete Boolean algebras. Moreover, $b\in\RO(M)$ iff $\neg b\in\RC(M)$.
\end{proposition}

\begin{remark}
Like for the complete Boolean algebra of regular open subsets of a topological space, for $S\subseteq \RO(M)$, we have $\bigvee_{\RO(M)} S = \s \diamond (\bigvee S)$, $\bigwedge_{\RO(M)} S = \s (\bigwedge S)$, and the complement of $a\in\RO(M)$ is given by $\s(\lnot a)$. The operations on $\RC(M)$ are defined similarly. 
\end{remark}

\begin{definition}
Let $M$ be an MT-algebra.
\begin{enumerate}[ref=\thedefinition(\arabic*)]
\item An element $a\in M$ is \emph{approximated from below by $\RO(M)$} if $a=\bigvee \{\square \diamond b \mid \diamond b \leq a\}$. Let $\GO(M)$ be the collection of such elements of $M$. 
\item An element $a\in M$ is \emph{approximated from above by $\RC(M)$} if $a=\bigwedge \{\diamond \square c \mid a \leq \square c\}$. Let $\GC(M)$ be the collection of such elements of $M$.\label [definition]{def:AC}

\end{enumerate}
\end{definition}

Observe that $\GO(M)\subseteq\O(M)$ and $\GC(M)\subseteq\C(M)$. 

\begin{definition} \label{def: T2}
We call $M\in\MT$ a {\em Hausdorff algebra} or a {\em $T_2$-algebra} if $\GC(M)$ join-generates~$M$. 
\end{definition}

\begin{remark}
Equivalently, $M$
is a Hausdorff algebra provided $\GO(M)$ meet-generates $M$. 
\end{remark}

Since $\GC(M)\subseteq\C(M)$, we have that every Hausdorff algebra is a $T_1$-algebra. Let $\MT_2$ be the full subcategory of $\MT_1$ consisting of Hausdorff algebras. \cref{nonspatial} shows that there are Hausdorff MT-algebras that are not spatial. Let $\SMT_2 = \MT_2 \cap \SMT$ and let $\Top_2$ be the full subcategory of $\Top_1$ consisting of Hausdorff spaces.

\begin{lemma} \label{lem:T_2}
Let $X \in \Top$ and $M\in\MT$. 
\begin{enumerate} [ref=\thelemma(\arabic*)] 
\item $X\in\Top_2$ iff $(\P(X),\int)\in\MT_2$.\label [lemma]{lem: T2}
\item $M\in\MT_2$ implies $\at(M)\in\Top_2$.
\end{enumerate}
\end{lemma}

\begin{proof}
(1) Observe that $X$ is a Hausdorff space iff $\{x\}=\bigcap {\{\overline U \mid x \in U \in \Omega(X)\}}$ for each $x \in X$. 
Since $\left \{\overline U \mid x \in U \in \Omega(X)\right \} = \left\{ \overline{\int A} \mid x \in \int A, A \in \P(X) \right\}$, 
 we obtain that $X\in\Top_2$ iff $ (\P(X),\int)\in\MT_2$.

(2) Let $x,y \in \at(M)$ with $x\ne y$. Since $M$ is a Hausdorff algebra, $x = \bigwedge \{ \diamond \square c \mid x \le \square c \}$. Because $y \not \leq x$, there is $c\in M$ such that $x \le \square c$ and $y\not\le \diamond\square c$. Let $a=\square c$ and $b=\neg\diamond\square c$. Then $a,b\in \O(M)$, $a \wedge b=0$, $x\le a$, and $y\le b$. Therefore, $\eta(a),\eta(b)$ are disjoint open subsets of $\at(M)$, $x\in\eta(a)$, and $y \in \eta(b)$. Thus, $\at(M)$ is a Hausdorff space. 
\end{proof}

As an immediate consequence of \cref{thm: duality for T1,lem:T_2}, we obtain:

\begin{theorem} \label{thm: adj and duality for T2}
    The dual adjunction between $\Top_1$ and $\MT_1$ restricts to a dual adjunction between $\Top_2$ and $\MT_2$, and the dual equivalence between $\Top_1$ and $\SMT_1$ restricts to a dual equivalence between $\Top_2$ and $\SMT_2$.
\end{theorem}

As in topological spaces (see, e.g., \cite[p.~15]{PicadoPultr2021}), we have that every Hausdorff MT-algebra is sober.

\begin{proposition} \label{prop: T2 implies sober}
  If $M \in \MT_2$, then $M \in \MT_{\mathrm{Sob}}$.
\end{proposition}

\begin{proof}
 
It suffices to show that if $c \neq 0$ is a closed element that is not an atom, then $c$ is not join-irreducible. 
Write $c = a \vee b$ with $0 < a, b < c$ and $a \wedge b = 0$. Since $M \in \MT_2$, there is a nonzero element in $\GC(M)$ under $a$. Without loss of generality we may assume that $a \in \GC(M)$, so $a = \bigwedge\{\diamond \square d \mid a \leq \square d\}$. Therefore, there is $d$ such that $a \leq \square d$ and $b \not\leq \diamond \square d$ (otherwise $b\leq a$). 
Let $m = \neg \square d$ and $n = \diamond \square d$. Both $m$ and $n$ are closed, and $m \vee n = \neg \square d \vee \diamond \square d = 1$, so $c \leq m \vee n$. However, $c \not\leq m$ because $a \leq \square d$, and $c \not\leq n$ because $b \not\leq \diamond \square d$. Thus, $c$ is not join-irreducible.
\end{proof}

We next relate $T_2$-algebras to Hausdorff frames. We recall (see, e.g., \cite[p.~330]{PicadoPultr2012}) that each element $a$ of a frame $L$ has the {\em pseudocomplement} $a^{*} = \bigvee \{ b \in L \mid b \wedge a =0 \}$. If $L=\O(M)$ for some MT-algebra $M$, then $a^{*}=\square\neg a$ (see, e.g., \cite[p.~18]{Esakia}).

\begin{definition}\cite[p.~43]{PicadoPultr2021}
A frame $L$ is {\em Hausdorff} if for all $a\in L \setminus \{1\}$ we have 
\[
a=\bigvee \{u \in L \mid u\leq a \text{ and } u^* \nleq a \}.
\]
\end{definition}

\begin{proposition} \label{lem: OM Hausdorff}
If $M\in\MT_2$, then $\O(M)$ is a Hausdorff frame.
\end{proposition}

\begin{proof}
Let $a \in \O(M)$ with $a\ne 1$.  
Since $M \in \MT_2$, there is $u \in \GO(M)$ such that $a\le u \ne 1$. Therefore, $u = \bigvee \{\square \diamond v \mid \diamond v \leq u \}$, and so $a = \bigvee \{ a\wedge\square \diamond v \mid \diamond v \leq u \}$. Let $b = a \wedge \square\diamond v$. Then $b\in\O(M)$ and $b\le a$. Suppose that $\square \neg b \leq a$. Then $\neg a \leq \diamond b \leq \diamond v \leq u$, a contradiction since $a, \neg a \leq u$ implies that $u=1$. Thus, $ \square \neg b \nleq a$.
Consequently, $a = \bigvee\{b \in \O(M) \mid  b \leq a \text{ and } \square \neg b \nleq a \}$, and hence $\O(M)$ is a Hausdorff frame.
\end{proof}

\begin{remark}

The converse of \cref{lem: OM Hausdorff} is not true. Indeed, it follows from \cite[p.~46]{PicadoPultr2021} that a Hausdorff frame is not necessarily subfit. Thus, if $\O(M)$ is such, then $M$ can't be a $T_2$-algebra by \cref{subfit}.
\end{remark}

Let $\HFrm$ be the full subcategory of $\Frm$ consisting of Hausdorff frames, and $\Frm_{\bf sfit}$ the full subcategory of $\HFrm$ consisting of subfit Hausdorff frames. 

\begin{proposition}\label{essentially surjective MT2}
The restriction $\O:\MT_2\to\HFrm_{\bf sfit}$ is well defined and faithful.
\end{proposition}

\begin{proof}
That the restriction $\O:\MT_2\to\HFrm_{\bf sfit}$ is well defined follows from \cref{subfit,lem: OM Hausdorff}, and that it is faithful from \cref{thm: O faithful for T0}. 
\end{proof}

We don't know whether the restriction of $\O$ to $\MT_2$ is essentially surjective or full. However, the situation improves in the spatial case, as we next see. 

\begin{lemma} \label{lem: spatial + T2 implies subfit}
Every spatial Hausdorff frame is subfit.
\end{lemma}

\begin{proof}
Let $L$ be a spatial Hausdorff frame. Then $L \cong \Omega(\pt(L))$, and so $\Omega(\pt(L))$ is a Hausdorff frame. Since $\pt(L)$ is sober and hence a $T_0$-space, it follows from \cite[p.~43]{PicadoPultr2021} that $\pt(L)$ is a Hausdorff space. But then $\Omega(\pt(L))$ is subfit (see, e.g., \cite[p.~24]{PicadoPultr2021}). Thus, $L$ is subfit.
\end{proof}

Let $\SHFrm$ be the full subcategory of $\HFrm$ consisting of spatial Hausdorff frames. By \cref{lem: spatial + T2 implies subfit}, 
$\SHFrm$ is a full subcategory of $\HFrm_{\bf sfit}$.

\begin{theorem} \label{thm: equivalence for spatial T2}
The restriction $\O:\SMT_{2}\to\bf SHFrm$ is an equivalence.
\end{theorem}

\begin{proof}
By \cref{thm: full}, the restriction $\O:\SMT_{\bf sob} \to \SFrm$ is an equivalence. Further restriction of $\O$ to 
Hausdorff MT-algebras is well defined by \cref{essentially surjective MT2}. Let $L \in \bf SHFrm$. Since $L$ is spatial, $L\cong\Omega(\pt(L))$. Letting $M$ be the powerset of $\pt(L)$, we have that $M$ is a spatial MT-algebra and $\O(M)=\Omega(\pt(L))$ is isomorphic to $L$. Moreover, since $\pt(L)$ is a $T_0$-space and $\Omega(\pt(L))$ is a Hausdorff frame, $\pt(L)$ is a Hausdorff space by \cite[p.~43]{PicadoPultr2021}. Therefore, $M$ is a Hausdorff MT-algebra by \cref{lem: T2}, and hence the restriction of $\O$ is essentially surjective. Finally, since every Hausdorff MT-algebra is sober (see \cref{prop: T2 implies sober}), we obtain  that the restriction of $\O$ is an equivalence between $\SMT_2$ and $\SHFrm$.
\end{proof}

We now turn our attention to regular MT-algebras. 
We start by recalling the definition of a regular frame (see, e.g., \cite[p.~89]{PicadoPultr2012}). For a frame $L$ and $a,b\in L$, we recall that $b$ is {\em rather below} or {\em well inside} $a$, written $b \prec a$, provided $b^{*}\vee a=1$. Then $L$ is {\em regular} if $\prec$ is approximating, meaning that $a=\bigvee \{b\in L \mid b \prec a\}$ for each $a \in L$. If $L$ is the frame of opens of a topological space $X$, then $U \prec V$ iff $\overline{U}\subseteq V$. This has an obvious generalization to the MT-algebra $(\P(X),\int)$, which further generalizes to an arbitrary MT-algebra.

\begin{definition} \label{def: well inside} 
Let $M$ be an MT-algebra. For $a , b \in M$ define $a \lhd b$ provided $\diamond a \leq \square b$.
\end{definition}

\begin{lemma}\label{regular implies regular}
Let $M$ be an MT-algebra. If $a,b \in \O(M)$, then $ a \lhd b \iff a \prec b$.
\end{lemma}

\begin{proof}
For $a,b \in \O(M)$, we have 
\[
a \lhd b \iff \diamond a \leq b \iff \square \neg a \vee b =1 \iff a^* \vee b =1 \iff a \prec b. \qedhere
\] 
\end{proof}

The proof of the following lemma is straightforward and we skip the details. 

\begin{lemma}\label{prec}
Let $M$ be an MT-algebra. 
\begin{enumerate}[ref=\thelemma(\arabic*)]
\item $a \lhd b \implies a \leq b$.
\item $0 \lhd a \lhd 1$. 
\item $x \leq a \lhd b \leq y \implies x \lhd y$. \label[lemma]{a prec b}
\item $a \lhd b \iff \neg b \lhd \neg a$.
\item $a_i \lhd b_i$ for $ i=1,2 \implies (a_1 \vee a_2) \lhd (b_1 \vee b_2)$ and $(a_1 \wedge a_2) \lhd (b_1 \wedge b_2)$.
\end{enumerate}
\end{lemma}

\begin{definition} \label{def: T3}
We call $M\in\MT$ a {\em regular algebra} or a {\em $T_3$-algebra} if $M$ is a $T_1$-algebra and for all $a \in \O(M)$ we have $a = \bigvee \{b \in \O(M) \mid b \lhd a\}$. 
\end{definition}

\begin{remark}\label{T3}
Equivalently, a $T_1$-algebra $M$ is a regular MT-algebra provided for all $c \in \C(M)$ we have $ c= \bigwedge \{d \in \C(M) \mid c \lhd d\}$.
\end{remark}

\begin{lemma}\label{AO(M)=O(M)} 
If $M$ is a regular MT-algebra, then $\GO(M)=\O(M)$ and $\GC(M)=\C(M)$.
\end{lemma}

\begin{proof}
Clearly $\GO(M) \subseteq \O(M)$. For the reverse inclusion, let $a\in\O(M)$. Since $M$ is a regular algebra, $a = \bigvee \{b \in \O(M) \mid b \lhd a\}$. Observe that $b\lhd a$ implies that $b \leq \square\diamond b \lhd a$. Therefore, $a=\bigvee \{\square \diamond b \mid \diamond b \leq a\}$, and hence $a \in \GO(M)$. Thus, $\GO(M)=\O(M)$. The proof that $\GC(M)=\C(M)$ is similar.
\end{proof}

\begin{proposition}
 
If $M$ is a regular MT-algebra, then $M$ is a Hausdorff MT-algebra.
\end{proposition}

\begin{proof}

Since $M$ is a $T_1$-algebra, $\C(M)$ join-generates $M$. By \cref{AO(M)=O(M)}, $\C(M)=\GC(M)$. Therefore, $\GC(M)$ join-generates $M$, and hence $M$ is a Hausdorff algebra.
\end{proof}

Let $\MT_3$ be the full subcategory of $\MT_2$ consisting of regular MT-algebras. \cref{nonspatial} shows that there are regular MT-algebras that are not spatial. Let $\SMT_3=\MT_3\cap\SMT$ and let $\Top_3$ be the full subcategory of $\Top_2$ consisting of regular spaces.

\begin{lemma} \label{lem: MT3 and Top3}
Let $X \in \Top$ and $M\in\MT$.
    \begin{enumerate}
    \item $X\in\Top_3$ iff $ (\P(X),\int)\in\MT_3$.
    \item $M\in\MT_3$ implies $\at(M)\in\Top_3$.
    \end{enumerate}
\end{lemma}

\begin{proof}
(1) Observe that $X$ is a regular space iff $X$ is a $T_1$-space and 
for every $U\in\Omega(X)$ and $x \in U$ there is $V\in\Omega(X)$ such that $x \in V \subseteq \overline{V} \subseteq U$. The latter condition is equivalent to $U = \bigcup\{V \in \Omega(X) \mid V \lhd U\}$ for each $U\in\Omega(X)$. Thus, by \cref{lem: T1}(1), $X\in\Top_3$ iff $(\P(X),\int)\in\MT_3$.

   (2) Since $M\in\MT_1$, we have $\at(M)\in\Top_1$ by \cref{lem: T1}(2). Let $a\in\O(M)$ and $x \in \eta(a)$. 
   Then $x \leq a$. Since $M\in\MT_3$, we have $a = \bigvee \{ b\in\O(M) \mid b \lhd a \}$. Therefore, 
   $x \leq b$ for some $b \in \O(M)$ such that $\diamond b \le a$. Thus, by \cref{rem: closed sets}, $x \in \eta(b) \subseteq \overline{\eta(b)} \subseteq \eta(\diamond b) \subseteq \eta(a)$.
   Consequently, $\at(M)\in\Top_3$.
\end{proof}

As an immediate consequence of \cref{thm: adj and duality for T2,lem: MT3 and Top3}, we obtain:

\begin{theorem} \label{thm: duality for T3}
    The dual adjunction between $\Top_2$ and $\MT_2$ restricts to a dual adjunction between $\Top_3$ and $\MT_3$, and the dual equivalence between $\Top_2$ and $\SMT_2$ restricts to a dual equivalence between $\Top_3$ and $\SMT_3$.
\end{theorem}

Let $\bf RegFrm$ be the full subcategory of $\Frm$ consisting of regular frames. Since each regular frame is Hausdorff (see, e.g., \cite[p.~91]{PicadoPultr2021}), $\bf RegFrm$ is a full subcategory of $\bf HFrm$.

\begin{proposition}\label{essentially surjective MT3}
The restriction $\O:\MT_3\to\bf RegFrm$ is well defined, essentially surjective, and faithful.
\end{proposition}

\begin{proof}

That the restriction $\O:\MT_3\to\bf RegFrm$ is well defined follows from \cref{regular implies regular}, and that it
is faithful follows from \cref{thm: O faithful for T0}. To see that it 
is essentially surjective, let $L \in \bf RegFrm$. By the proof of \cref{thm: essentially surj}, $\O\left(\overline{B(L)}\right)=L$. Since $L$ is a regular frame, using \cref{regular implies regular} again, $a = \bigvee\{b \in \O(M) \mid b \lhd a \}$ for every $a \in \O\left(\overline{B(L)}\right)$. It remains to see that $\overline{B(L)}$ is a $T_1$-algebra. By \cref{T_1/2}, $\overline{B(L)}$ is a $T_{1/2}$-algebra. Let $a\in\overline{B(L)}$. Then $a = \bigvee S$ for some set $S$ of locally closed elements of $\overline{B(L)}$. Let $s \in S$. Then $s = u \wedge \neg v$ for some $u , v \in \O\left(\overline{B(L)}\right)$. Therefore,  
\[
u=\bigvee \{b \in \O(M) \mid \diamond b \leq u\}=\bigvee \{\diamond b \mid b\in\O(M), \diamond b \leq u\}.
\] Thus, $s=\bigvee \{\diamond b \wedge \neg v \mid b,v\in\O(M), \diamond b \leq u\}$. Consequently, every element of $S$ is a join of closed elements, so $a$ is a join of closed elements, and hence $\overline{B(L)}$ is a $T_1$-algebra.  
\end{proof}

While it remains open whether $\O:\MT_3\to\bf RegFrm$ is full, we conclude this section by showing that it is indeed an equivalence when restricted to the corresponding spatial categories. Let $\bf SRegFrm$ be the full subcategory of $\bf RegFrm$ consisting of spatial regular frames.

\begin{theorem} \label{thm: regular}
The restriction $\O:\SMT_3\to\bf SRegFrm$ is an equivalence.
\end{theorem}

\begin{proof}

By \cref{thm: equivalence for spatial T2}, $\O:\SMT_2 \to \bf SHFrm$ is an equivalence. Now apply
\cref{essentially surjective MT3}. 
  \end{proof}

\section{Completely regular and normal MT-algebras}\label{sec:8}

In this final section, we extend the concepts of completely regular and normal spaces to MT-algebras. We show that a $T_1$-algebra $M$ is completely regular (resp.~normal) iff $\O(M)$ is a completely regular (resp.~normal) frame. We also develop a version of Urysohn's lemma for MT-algebras to show that every normal MT-algebra is completely regular, and restrict the dual adjunctions and dual equivalences of the previous section to the setting of completely regular and normal MT-algebras.

We start by recalling the definition of a completely regular frame (see, e.g., \cite[p.~91]{PicadoPultr2012}). Let $L$ be a frame and $a,b\in L$. Then $b$ is {\em completely below} $a$, written $b \prec \prec a$, provided there is a family $\{ c_p \mid p \in [0,1]\cap\mathbb Q \} \subseteq L$
such that $b \leq c_0$, $c_1 \leq a$, and $p < q \Rightarrow c_p \prec c_q$. We call $L$ {\em completely regular} if $\prec \prec$ is approximating, meaning that $a=\bigvee \{b \in L \mid b \prec \prec a\}$ for each $a \in L$. Clearly $\prec \prec$ is a subrelation of $\prec$, and hence every completely regular frame is regular. 

We generalize the notion of a completely regular frame to MT-algebras. 

\begin{definition} 
Let $M\in\MT$. For $a , b \in M$ define $b \ {\lhd\lhd} \ a$ 
provided there is a family $\{ c_p \}\subseteq M$, where $p \in [0,1]\cap\mathbb Q$,  
such that $b \leq c_0$, $c_1 \leq a$, and $p < q$ implies $c_p \lhd c_q$.
\end{definition}

\begin{remark} \label{rem: comp reg}
Clearly ${\llcurly}$ is a subrelation of $\lhd$. Moreover, we can always choose the $c_p$ to be either open or closed.
\end{remark}

As an immediate consequence of \cref{prec}, we obtain: 

\begin{lemma} \label{lem: completely below}
Let $M$ be an MT-algebra. 
\begin{enumerate}
\item $a\llcurly b \implies a \leq b$.
\item $0\llcurly a\llcurly 1$.
\item $x \leq a\llcurly b \leq y \implies x\llcurly y$.
\item $a\llcurly b \iff \neg b\llcurly \neg a$.
\item $a_i\llcurly b_i$ for $ i=1,2 \implies (a_1 \vee a_2)\llcurly (b_1 \vee b_2)$ and $(a_1 \wedge a_2)\llcurly (b_1 \wedge b_2)$.
\end{enumerate}
\end{lemma}

\begin{definition} \label{def: T3 and half}
We call an MT-algebra $M$ a {\em completely regular algebra} or a {\em $T_{3\frac{1}{2}}$-algebra} if $M$ is a $T_1$-algebra and for all $a \in \O(M)$ we have $a = \bigvee \{b \in \O(M) \mid b \llcurly a\}$. 
\end{definition}

\begin{remark}\label{T 3(1/2)}
Equivalently, a $T_1$-algebra $M$ is a completely regular algebra provided for all $c \in \C(M)$ we have $c= \bigwedge \{d \in \C(M) \mid c \llcurly d\}$.
\end{remark}

\begin{lemma}\label{opens replaced by regular opens}
Let $M$ be an MT-algebra.
\begin{enumerate}
\item If $a\in\O(M)$ and $b\llcurly a$, then $\square \diamond b\llcurly a$.
\item $a=\bigvee \{b \in \RO(M) \mid b \llcurly a \}$ for each $a \in \O(M)$.
\item $c=\bigwedge \{d \in \RC(M) \mid c \llcurly d \}$ for each $c \in \C(M)$.
\end{enumerate}
\end{lemma}

\begin{proof}
(1) 
By \cref{rem: comp reg}, there is a family $\{ c_p \in \C(M) \mid p \in [0,1]\cap\mathbb Q \}$ such that $b \leq c_0$, $c_1 \leq a$, and $p < q$ implies $c_p \lhd c_q$. Therefore, $\square \diamond b \leq \square c_0, \square c_1 \leq a, \text{ and }  p < q \Rightarrow \diamond \square c_p \leq c_p \leq \square c_q$. Thus, 
the family $\{ \square c_p \in \O(M) \mid p \in [0,1]\cap\mathbb Q \}$ witnesses that 
$\square \diamond b\llcurly a$.

(2) follows from (1) and (3) follows from (2).
\end{proof}

As an immediate consequence of \cref{regular implies regular,opens replaced by regular opens}, we obtain:

\begin{theorem} \label{thm: cmpl reg}
A $T_1$-algebra $M$ is completely regular iff $\O(M)$ is a completely regular frame.
\end{theorem}

Also, as an immediate consequence of \cref{rem: comp reg}, we get:

\begin{lemma} \label{lem: T3 and half implies T3}
If $M$ is a completely regular MT-algebra, then $M$ is a regular MT-algebra.
\end{lemma}

Let $\MT_{3\frac{1}{2}}$ be the full subcategory of $\MT_3$ consisting of completely regular MT-algebras. \cref{nonspatial} shows that there are completely regular MT-algebras that are not spatial. We let $\SMT_{3\frac{1}{2}} = \MT_{3\frac{1}{2}} \cap \SMT$ and $\Top_{3\frac{1}{2}}$ be the full subcategory of $\Top_3$ consisting of completely regular spaces.

\begin{lemma} \label{lem: 3 and 1/2}
Let $ X \in \Top$ and $M\in\MT$. 
\begin{enumerate}
\item $X \in \Top_{3\frac{1}{2}}$ iff $ (\P(X),\int) \in \MT_{3\frac{1}{2}}$.
\item $M\in\MT_{3\frac{1}{2}}$ implies $\at(M) \in \Top_{3\frac{1}{2}}$.
\end{enumerate}
\end{lemma}

\begin{proof}
(1)  $X$ is a $T_{3\frac{1}{2}}$-space iff $X$ is a $T_1$-space and $U=\bigcup\{V\in \Omega(X) \mid V \prec \prec U\}$ for each $U \in \O(X)$ (see, e.g., \cite[p.~91]{PicadoPultr2012}). By \cref{thm: cmpl reg}, $X\in\Top_{3\frac{1}{2}}$ iff $(\P(X),\int) \in\MT_{3\frac{1}{2}}$.

 (2) By \cref{lem: T1}(2), $\at(M)\in\Top_1$. Let $a \in \O(M)$ and $x\in\eta(a)$. Then $x \leq a$. Since $M$ is completely regular, $a = \bigvee \{ b \in \O(M) \mid b \llcurly a \}$. Therefore, $x\le b$ for some $b\in\O(M)$ with $b \llcurly a$. The latter implies that there are $c_r \in \O(M)$ such that $b \le c_0$, $c_1 \le a$, and $p<q$ implies $c_p \lhd c_q$, so $\diamond c_p \leq \square c_q$. Thus, $\eta(b)\subseteq\eta(c_0)$, $\eta(c_1)\subseteq\eta(a)$, and by \cref{rem: closed sets}, $\overline{\eta(c_p)} \subseteq \eta(\diamond c_p) \subseteq \eta(\square c_q) \subseteq \int(\eta(c_q))$ for $ p < q$. Consequently, $x\in\eta(b)$ and $\eta(b)\prec \prec \eta(a)$, yielding that $\eta(a) = \bigcup \{ \eta(b) \in \Omega(\at(M)) \mid \eta(b) \prec \prec \eta(a)\}$. Hence, $\at(M) \in \Top_{3\frac{1}{2}}$.
 \end{proof}
 
 As an immediate consequence of \cref{thm: duality for T3,lem: 3 and 1/2}, we obtain:
 
\begin{theorem} \label{thm: duality for T3 and half}
    The dual adjunction between $\Top_3$ and $\MT_3$ restricts to a dual adjunction between $\Top_{3\frac{1}{2}}$ and $\MT_{3\frac{1}{2}}$, and the dual equivalence between $\Top_3$ and $\SMT_3$ restricts to a dual equivalence between $\Top_{3\frac{1}{2}}$ and $\SMT_{3\frac{1}{2}}$.
\end{theorem}

Let $\bf CRegFrm$ be the full subcategory of $\bf RegFrm$ consisting of completely regular frames.

\begin{proposition}\label{CRegFrm}
The restriction $\O:\MT_{3\frac{1}{2}}\to\bf CRegFrm$ is well defined, essentially surjective, and faithful.
\end{proposition}

\begin{proof}
That $\O:\MT_{3\frac{1}{2}}\to\bf CRegFrm$ is well defined follows from \cref{thm: cmpl reg} and that it is faithful from \cref{thm: O faithful for T0}. To see that $\O$ 
is essentially surjective, let $L \in \bf CRegFrm$. By the proof of \cref{thm: essentially surj}, $\O\left(\overline{B(L)}\right)=L$. Since $\O\left(\overline{B(L)}\right)$ is a completely regular frame, it is a regular frame, so $\overline{B(L)}$ is a $T_1$-algebra by the proof of \cref{essentially surjective MT3}. Applying \cref{thm: cmpl reg} again yields that 
$\overline{B(L)}$ is a completely regular MT-algebra, completing the proof.  
\end{proof}

It is unclear whether $\O:\MT_{3\frac{1}{2}}\to\bf CRegFrm$ is full. Let $\bf SCRegFrm$ be the full subcategory of $\bf CRegFrm$ consisting of spatial 
frames. By \cref{thm: regular,CRegFrm}, we have:

\begin{theorem} \label{thm: completely reg equiv}
The restriction $\O:\SMT_{3\frac{1}{2}}\to\bf SCRegFrm$ is an equivalence.
\end{theorem}

Finally, we turn our attention to normal MT-algebras. For this we recall 
(see, e.g., \cite[p.~91]{PicadoPultr2012}) that a frame $L$ is {\em normal} if for all $a,b \in L$ with $a \vee b=1$ there are $u , v \in L$ such that $u \wedge v=0$ and $a \vee v=1=b \vee u$. If $L$ is the frame of opens of a $T_1$-space $X$, then $L$ is normal iff $X$ is a normal space (see, e.g., \cite[p.~138]{PicadoPultr2021}).
This has an obvious generalization to $(\P(X),\int)$, which we will further generalize to arbitrary MT-algebras.

\begin{definition} \label{def: T4}
Let $M\in\MT$. We call $M$ a {\em normal algebra} or a {\em $T_4$-algebra} if $M$ is a $T_1$-algebra and for $c, d \in \C(M)$, from $c \wedge d = 0$ it follows that there exist $a, b \in \O(M)$ such that $a \wedge b = 0$, $c \leq a$, and $d \leq b$.
\end{definition}

\begin{remark}\label{normal}
Equivalently, a $T_1$-algebra $M$ is a normal MT-algebra provided for all $a,b \in \O(M)$ with $a \vee b=1$ there exist $c , d \in \C(M)$ such that $c \vee d=1$, $c \leq a$, and $d \leq b$.
\end{remark}

As an immediate consequence, we obtain:

\begin{theorem} \label{thm: normality}
A $T_1$-algebra $M$ is normal iff $\O(M)$ is normal.
\end{theorem}
 
The well-known characterization of normal spaces (see, e.g., \cite[p.~40]{Engelking}) 
readily generalizes to MT-algebras:

\begin{lemma}\label{open and closed element inbetween}
Let $M$ be a $T_1$-algebra. Then 
 $M$ is a normal algebra iff for all $a\in\O(M)$ and $c\in\C(M)$, from $c\le a$ it follows that there exist $b\in\O(M)$ and $d\in\C(M)$ such that $c\le b\le d\le a$. 
\end{lemma}

\begin{theorem} \label{MT4 implies MT3}
Every normal MT-algebra is regular.
\end{theorem}

\begin{proof}
It is sufficient to show that $a = \bigvee \{b \in \O(M) \mid b \lhd a\}$ for each $a\in\O(M)$. Since $a\in\O(M)$, we have $\neg a\in\C(M)$. Because $M$ is a $T_1$-algebra, $\neg a = \bigwedge \{ u \in \O(M) \mid \neg a \le u \}$. Since $M$ is a normal algebra,  \cref {open and closed element inbetween} implies that for each such $u$ there exist $v\in\O(M)$ and $d\in\C(M)$ such that $\neg a\le v\le d\le u$. Therefore, $\neg u \le \neg d \le \neg v \le a$. Letting $\neg d=b$, we obtain that $b\in\O(M)$ and $b \ \lhd \ a$. Thus,
\[
a = \bigvee \{ \neg u \mid u \in \O(M), \neg u \le b \lhd a\} = \bigvee \{b \in \O(M) \mid b \lhd a\},
\] 
as required. 
\end{proof}

Let $\MT_{4}$ be the full subcategory of $\MT_3$ consisting of normal MT-algebras. \cref{nonspatial} shows that there are normal MT-algebras that are not spatial. Let $\SMT_{4} = \MT_{4} \cap \SMT$ and let $\Top_{4}$ be the full subcategory of $\Top_3$ consisting of normal spaces.

\begin{lemma} \label{lem: T4}
Let $ X \in \Top$ and $M\in\MT$. Then $X \in \Top_4$ iff $ (\P(X),\int) \in \MT_4$.
\end{lemma}

\begin{proof}
This follows from \cref{lem: T1}(1) and the definition of normality in MT-algebras and topological spaces.
 \end{proof}

Since normality is not a hereditary property, and hence a relativization of a normal MT-algebra may not be normal, unlike other separation axioms, it is unclear whether $M \in \MT_4$ implies that $\at(M) \in \Top_4$. 
Therefore, it is unclear whether the dual adjunction of \cref{thm: duality for T3} restricts to $\MT_4$ and $\Top_4$. Nevertheless, \cref{lem: T4} yields that the dual equivalence of \cref{thm: duality for T3} does restrict to $\SMT_4$ and $\Top_4$:

\begin{theorem} \label{thm: duality for T4}
    The dual equivalence between $\Top_3$ and $\SMT_3$ restricts to a dual equivalence between $\Top_{4}$ and $\SMT_{4}$.
\end{theorem}

It is a consequence of the celebrated Urysohn lemma that every normal space is completely regular. To generalize this result to MT-algebras, we require a version of Urysohn's lemma for MT-algebras, which is our next goal. We start with the following lemma, which generalizes a similar result for spaces 
(see, e.g., \cite[Thms.~2.11, 3.9(vi)]{Naimpally}).

\begin{lemma} \label{lemma: interpolation}
    Let $M$ be a normal MT-algebra. If $a,b\in M$ with $a \lhd b$, then there is $u \in \O(M)$ such that $a \lhd u \lhd b$.
\end{lemma}

\begin{proof}
Since $a \ \lhd \ b$, we have $\diamond a \leq \square b$. Because $M$ is a normal algebra, by \cref{open and closed element inbetween} there are $u\in\O(M)$ and $c\in\C(M)$ such that $\diamond a \leq u \le c \le \square b$ . Thus, $a \lhd u$ and $u \lhd b$. 
\end{proof}

We are ready to prove a version of Urysohn's lemma for MT-algebras. The proof is a direct adaptation of the proof of the first part of Urysohn's lemma 
(see, e.g., \cite[p.~41]{Engelking}). 

\begin{lemma} [Urysohn's lemma for MT-algebras] \label{creating up}
 Let $M$ be a normal MT-algebra, $a \in \O(M)$, $c \in \C(M)$, and $c \le a$. Then there is a family $\mathcal U = \{u_p \in \O(M) \mid p \in \mathbb Q \cap[0,1] \}$ such that

 \begin{enumerate}
     \item $c \leq u_0$ and $u_1 \le a$.
     \item $p < q$ implies $u_p \lhd u_q$ for all $ p,q \in \Q \cap [0,1]$.
 \end{enumerate}
\end{lemma}

\begin{proof}
   Since $\Q \cap (0, 1)$ is countable, let $\{p_n \mid  n \in \N\}$ be an enumeration of $\Q \cap (0,1)$ such that  
   $p_0 = 0$ and $p_1 = 1$. We prove by induction that for each $n \ge 1$ there is $\mathcal U_n = \{u_{p_0}, \dots, u_{p_n}\}$ such that \begin{enumerate}[(i)]
     \item $c \leq u_0$ and $u_1 = a$;
     \item For all $i,j \in \N$ with $i,j \leq n$ we have $p_i < p_j \implies u_i \lhd u_j$.
 \end{enumerate}
 
Let $n =1$. Observe that $c = \diamond c \leq \square a = a$, so $c \lhd a$. By \cref{lemma: interpolation}, there is $u_0 \in \O(M)$ with $c \lhd u_0 \lhd a$.     Letting $u_1 = a$, we have that $\mathcal U_1 = \{u_0, u_1\}$ satisfies (i) and (ii).

Now suppose (i) and (ii) hold for $\mathcal U_n = \{u_{p_0}, \dots, u_{p_n}\}$. To simplify notation, we let $P = \{p_0, p_1, \dots, p_n\}$ and $r = p_{n+1}$.
We also let $p$ be an immediate predecessor and $q$ and immediate successor of $r$ in $P$.
By the inductive hypothesis, $u_p \lhd u_q$ . By \cref{lemma: interpolation}, there is $u_{p_{n+1}} \in \O(M)$ with $u_p \lhd u_{p_{n+1}} \lhd u_q$.
Thus, $\mathcal U_{n+1} := \mathcal U_n \cup \{u_{p_{n+1}}\}$ satisfies (i) and (ii).

Finally, letting $\mathcal U = \bigcup_{n \in \N} \mathcal U_n$ yields the desired $\mathcal U$.
\end{proof}

\begin{lemma}\label{coincide}
Let $M\in\MT_4$ and $a,b\in M$. Then $a \llcurly b$ iff $a \ \lhd \ b$. 
\end{lemma}

\begin{proof} 
Since ${\llcurly} $ is a subrelation of $\lhd$, we have that $a \llcurly b$ implies $a \ \lhd \ b$. Conversely, suppose that $a \ \lhd \ b$. Then $\diamond a \leq \square b$. By \cref{creating up}, there is a family of $u_p \in \O(M)$, with $p\in\Q\cap[0,1]$, such that $\diamond a \leq u_0$, $u_1 \leq \square b$, and $p < q$ implies $u_p \lhd u_q$. Therefore, $a \llcurly b$. 
\end{proof}

\begin{theorem}
    If $M$ is a normal MT-algebra, then $M$ is completely regular.
\end{theorem}

\begin{proof}
Let $M$ be a normal MT-algebra. Then $\lhd$ coincides with ${\llcurly}$ by \cref{coincide}. Therefore, the result follows from \cref{MT4 implies MT3}.
\end{proof}

Let $\bf NormFrm$ be the full subcategory of $\Frm_{\bf sfit}$ consisting of normal frames. Then $\bf NormFrm$ is a full subcategory of $\bf CRegFrm$ (see, e.g., \cite[p.~92]{PicadoPultr2012}).

\begin{proposition}\label{NFrm}
The restriction $\O:\MT_4 \to \bf NormFrm$ is well defined, essentially surjective, and faithful.
\end{proposition}

\begin{proof}
Let $M\in \MT_4$. By \cref{thm: normality}, $\O(M)$ is a normal frame. Since $M$ is a $T_1$-algebra, $\O(M)$ is also subfit by \cref{subfit}. Therefore, the restriction $\O:\MT_4 \to\bf NormFrm$ is well defined. It is faithful by \cref{thm: O faithful for T0} since $T_4$-algebras are $T_0$-algebras. To see that $\O$ 
is essentially surjective, let $L \in \bf NormFrm$. By the proof of \cref{thm: essentially surj}, $\O\left(\overline{B(L)}\right)=L$.  
By \cref{T_1/2}, $\overline{B(L)}$ is a $T_{1/2}$-algebra, so $\overline{B(L)}$ is a $T_1$-algebra by \cref{Sfit}. Thus, $\overline{B(L)}$ is a normal algebra by \cref{thm: normality}, completing the proof.
\end{proof}

We don't know whether $\O:\MT_4\to\bf NormFrm$ is full. However, if $\bf SNormFrm$ denotes the full subcategory of $\bf NormFrm$ consisting of spatial frames, we have:

\begin{theorem}
The restriction $\O:\SMT_4\to\bf SNormFrm$ is an equivalence.
\end{theorem}

\begin{proof}
Apply \cref{{thm: completely reg equiv},NFrm}.  
\end{proof}
In the diagram below we summarize the adjunctions, equivalences, and dual equivalences obtained in this paper. The notation $\hookrightarrow$ stands for being a reflective subcategory, with 
\begin{tikzpicture}
\draw[->] (0,.05) to [bend left] (.5,.05);
\end{tikzpicture}
being the corresponding left adjoint. We also use 
$\longleftrightarrow$ for dual equivalence  
and $\leq$ for being a full subcategory.\\

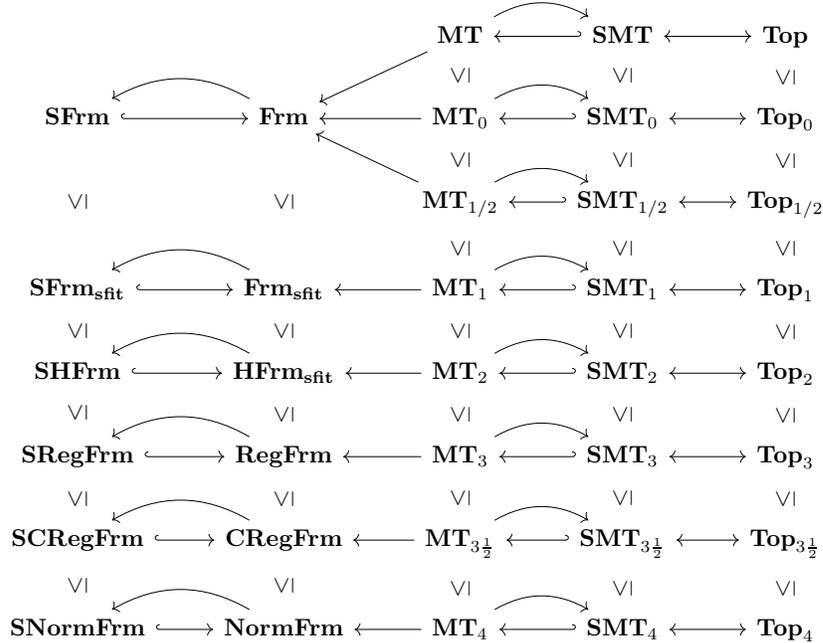
\adjustbox{scale=.9125,center}{
\begin{tikzcd}
    && \ar[ld] \MT \ar[r, hookleftarrow]
        & \SMT  \ar[r, <->] \ar[l, bend right, <-]
        & \Top\\
     \SFrm \ar[r, hookrightarrow]  
        & \Frm \ar[l, bend right]
        & \ar[l]\MT_0  \ar[r, hookleftarrow]\ar[u, symbol=\leq]
        & \SMT_0  \ar[r, <->] \ar[u, symbol=\leq] \ar[l, bend right, <-]
        & \Top_0 \ar[u, symbol=\leq]\\
    && \ar[lu]\MT_{1/2} \ar[r, hookleftarrow]\ar[u, symbol=\leq]
        & \SMT_{1/2}  \ar[r, <->] \ar[u, symbol=\leq] \ar[l, bend right, <-]
        & \Top_{1/2} \ar[u, symbol=\leq]\\
   \SFrm_{\bf sfit} \ar[r, hookrightarrow] \ar[uu, symbol=\leq]
       & \ar[uu, symbol=\leq] \Frm_{\bf sfit}  \ar[l, bend right]
       & \MT_1\ar[l] \ar[r, hookleftarrow]\ar[u, symbol=\leq]
       & \SMT_1  \ar[r, <->] \ar[u, symbol=\leq] \ar[l, bend right, <-]
       & \Top_1 \ar[u, symbol=\leq]\\
   \SHFrm \ar[u, symbol=\leq] \ar[r, hookrightarrow]
        & \ar[u, symbol=\leq] \HFrm_{\bf sfit} \ar[l, bend right]
        & \ar[l]\MT_2 \ar[r, hookleftarrow]\ar[u, symbol=\leq]
        & \SMT_2  \ar[r, <->] \ar[u, symbol=\leq] \ar[l, bend right, <-]
        & \Top_2 \ar[u, symbol=\leq]\\
    \SRegFrm \ar[u, symbol=\leq] \ar[r, hookrightarrow]
        & \ar[u, symbol=\leq]\RegFrm \ar[l, bend right]
        & \ar[l]\MT_3 \ar[r, hookleftarrow]\ar[u, symbol=\leq]
        & \SMT_3  \ar[r, <->] \ar[u, symbol=\leq] \ar[l, bend right, <-]
        & \Top_3 \ar[u, symbol=\leq]\\
    \SCRegFrm \ar[u, symbol=\leq] \ar[r, hookrightarrow]
        & \ar[u, symbol=\leq]\CRegFrm  \ar[l, bend right]
        & \MT_{3\frac{1}{2}}\ar[l] \ar[r, hookleftarrow]\ar[u, symbol=\leq]
        & \SMT_{3\frac{1}{2}}  \ar[r, <->] \ar[u, symbol=\leq] \ar[l, bend right, <-]
        & \Top_{3\frac{1}{2}} \ar[u, symbol=\leq]\\
    \SNFrm \ar[u, symbol=\leq] \ar[r, hookrightarrow]
        &  \ar[u, symbol=\leq] \NFrm \ar[l, bend right]
        & \ar[l]\MT_4 \ar[r, hookleftarrow]\ar[u, symbol=\leq]
        & \SMT_4  \ar[r, <->] \ar[u, symbol=\leq] \ar[l, bend right, <-]
        & \Top_4 \ar[u, symbol=\leq]\\
\end{tikzcd}
}

We didn't include $\MT_{\bf Sob}$ in the diagram above since it is not comparable with $\MT_{1/2}$ and $\MT_1$, but we have the following version of the diagram that includes $\MT_{\bf Sob}$.\\ 
\begin{center}
\begin{tikzcd}
    && \ar[ld] \MT_0 \ar[r, hookleftarrow]
		& \SMT_0  \ar[r, <->] \ar[l, <-, bend right]
		& \Top_0 \\
	\SFrm \ar[r, hookrightarrow]
		& \Frm \ar[l, bend right]
		& \ar[l] \MT_{\bf Sob} \ar[r, hookleftarrow] \ar[u, symbol=\leq]
		& \SMT_{\bf Sob} \ar[r, <->] \ar[u, symbol=\leq] \ar[l, <-, bend right]
		& \Sob \ar[u, symbol=\leq]\\
	\SHFrm \ar[u, symbol=\leq]\ar[r, hookrightarrow]
		& \ar[u, symbol=\leq] \HFrm_{\bf sfit}  \ar[l, bend right]
		& \ar[l] \MT_{2} \ar[r, hookleftarrow] \ar[u, symbol=\leq] 
		& \SMT_{2} \ar[r, <->] \ar[u, symbol=\leq] \ar[l, <-, bend right]
		& \Top_{2} \ar[u, symbol=\leq]\\
\end{tikzcd}
\end{center}

The table below summarizes various separation axioms for MT-algebras studied in the paper. 

\begin{center}
\setlength{\tabcolsep}{12pt}
\begin{tabular}{ll}
    \toprule
    \bf MT-Algebras & \bf Condition \\ \midrule
    $T_0$ & $\WLC(M)$ join-generates $M$ (Def.~\ref{def:T0}). \\ 
    $T_{1/2}$ & $\LC(M)$ join-generates $M$ (Def.~\ref{def: T half}). \\ 
    $T_1$ & $\C(M)$ join-generates $M$ (Def.~\ref{def: T1}). \\
    $T_2$ & $\GC(M)$ join-generates $M$ (Def.~\ref{def: T2}).\\ 
    $T_3$ & $T_1 + a=\bigvee\{ b\in \O(M) \mid u \lhd a\}$ for all $a \in \O(M)$ (Def.~\ref{def: T3}). \\
    $T_{3\frac{1}{2}}$ & $T_1 + a=\bigvee\{ b \in \O(M) \mid u\llcurly a\}$ for all $a \in \O(M)$ (Def.~\ref{def: T3 and half}). \\
    $T_4$ & $T_1 + \forall  c ,d \in \C(M) , c \wedge d=0 \Longrightarrow$  \\
    & $\phantom{T_1 + {}}  \exists a, b\in \O(M) : a \wedge b =0$ and $c \leq a, d \leq b$ (Def.~\ref{def: T4}). \\
    \bottomrule
\end{tabular}
\end{center}

\section{Conclusions}

In this paper we introduced MT-algebras as an alternative pointfree approach to topology. One advantage of MT-algebras is that the entire $\Top$ is (contravariantly) embedded into $\MT$, while only $\Sob$ is (contravariantly) embedded into $\Frm$. We showed that the functor $\O:\MT \to \Frm$ is essentially surjective, and that it becomes full or faithful on several subcategories of $\MT$.

We also initiated a detailed study of separation axioms for MT-algebras that generalize the classic separation axioms in topology. They all 
are characterized by identifying appropriate subsets of an MT-algebra $M$ that join-generate $M$. We showed that the restriction $\O:\MT_0 \to \Top_0$ is faithful and studied when the corresponding spaces of $M$ and $\O(M)$ are homeomorphic, yielding that  the restriction $\O:\SMT_{\bf Sob} \to \SFrm$ is an equivalence. 

We showed that an MT-algebra $M$ is a $T_{1/2}$-algebra iff $M$ is isomorphic to $\overline{B(\O(M))}$ and proved that a sober $T_{1/2}$-algebra $M$ is spatial iff $\O(M)$ is a spatial frame. In addition, we proved that an MT-algebra $M$ is $T_1$ iff $M$ is $T_{1/2}$ and $\O(M)$ is a subfit frame. Nevertheless, the restriction $\O:\MT_1\to\Frm_{\bf sfit}$ is not full.

Restricting to Hausdorff MT-algebras, we have that 
$\O:\MT_2 \to \HFrm_{\bf sfit}$ is essentially surjective and faithful, and further restriction to the corresponding spatial categories yields an equivalence because Hausdorff MT-algebras are sober. This remains true for regular and completely regular MT-algebras, but care is needed when restricting to normal MT-algebras because the relativization of a normal MT-algebra may no longer be normal. However, the restriction $\O:\SMT_4 \to \SNFrm$ remains an equivalence. In addition, we do have an analogue of Urysohn's lemma for MT-algebras, which allows us to prove that $\MT_4$ is a full subcategory of $\MT_{3\frac{1}{2}}$. 

In conclusion, for ease of reference, we gather together several open problems that are scattered throughout the paper:

\begin{enumerate}
\item By \cref{thm: when is M spatial}, a sober $T_{1/2}$-algebra $M$ is spatial iff $\O(M)$ is a spatial frame. It is open whether the $T_{1/2}$ assumption is necessary. In other words, if $M$ is a sober MT-algebra, is it true that $M$ is spatial iff $\O(M)$ is spatial (see \cref{rem: spatiality of M vs O(M)})?

\item It follows from \cite[p.~43]{PicadoPultr2021} that if $M$ is a spatial $T_0$-algebra, then $M$ is Hausdorff iff $\O(M)$ is a Hausdorff frame. It is open whether the spatiality assumption can be dropped from the theorem. Related to this, it remains open whether the restriction $\O:\MT_2 \to \HFrm_{\bf sfit}$ is essential surjective or full.
\item By \cref{rem: O not full on MT1}, the restriction $\O:\MT_1 \to \Frm_{\bf sfit}$ is not full. However, it remains open whether the restrictions $\O:\MT_3 \to \RegFrm$, $\O:\MT_{3\frac{1}{2}} \to \CRegFrm$, or $\O:\MT_4 \to \NFrm$ are full.
\item It remains open whether $M\in\MT_{4}$ implies that $\at(M) \in \Top_{4}$. As a result, it remains open whether the dual adjunction between $\Top_3$ and $\MT_3$ restricts to a dual adjunction between $\Top_4$ and $\MT_4$.
\end{enumerate}

\bibliographystyle{abbrv}
\bibliography{references}

\end{document}